\documentclass[11pt]{amsart}

\usepackage{amssymb, amscd}
\usepackage{color}

\usepackage{euscript}

\usepackage{soul}
\usepackage{mathrsfs}
\numberwithin{equation}{section}

\DeclareSymbolFont{SY}{U}{psy}{m}{n}
\DeclareMathSymbol{\emptyset}{\mathord}{SY}{'306}

\newcommand{\overbar}[1]{\mkern 1.5mu\overline{\mkern-1.5mu#1\mkern-1.5mu}\mkern 1.5mu}

\theoremstyle{plain}

\newtheorem{thm}{Theorem}[section]
\newtheorem{cor}[thm]{Corollary}
\newtheorem{lem}[thm]{Lemma}
\newtheorem{prop}[thm]{Proposition}

\theoremstyle{definition}
\newtheorem{defn}[thm]{Definition}

\newtheorem{rem}[thm]{Remark}

\DeclareMathOperator{\Ad}{Ad}
\DeclareMathOperator{\ad}{ad}

\DeclareMathOperator{\Hol}{Hol}
\DeclareMathOperator{\Hom}{Hom}
\DeclareMathOperator{\Aut}{Aut}

\newcommand{\D}{ \ensuremath{{D}}}

\newcounter{defcounter}
\title[Homogeneous bundles and operators in the Cowen-Douglas class]
{Homogeneous Hermitian holomorphic vector bundles and \\the Cowen-Douglas class over bounded symmetric domains 
}

\author[A. Kor\'{a}nyi]{Adam Kor\'{a}nyi} 
\author[G. Misra]{Gadadhar Misra}

\address[A. Kor\'{a}nyi]{Lehman College\\
Bronx, NY 10468}
\address[G. Misra]{Indian Institutte of Science\\Bangalore 560012}
\email[A. Kor\'{a}nyi]{adam.koranyi@lehman.cuny.edu}
\email[G. Misra]{gm@math.iisc.ernet.in}

\thanks{Both the authors were supported, in part, by a DST - NSF S\&T Cooperation Program and the J C Bose National Fellowship of the Department 
of Science and Technology. The second author also gratefully acknowledges the support from the University Grants Commission Centre for Advanced Studies.}

\begin{document}

\begin{abstract}
It is known that all the vector bundles of the title can be obtained by holomorphic induction from representations of a certain parabolic Lie algebra on finite dimensional inner product spaces. The representations, and the induced bundles, have composition series with irreducible factors.  Our first   main result is the construction of  an explicit differential operator intertwining the  bundle with the direct sum of its factors. Next, we study Hilbert spaces of sections of  these bundles. We use this to get,  in particular, a full description and a similarity theorem for homogeneous  $n$-tuples of operators in the Cowen-Douglas class of the Euclidean unit ball in $\mathbb C^n$.

\end{abstract}
\maketitle
\setcounter{section}{-1}
 \section{Introduction}
A domain in $\mathbb C^n$ is said to be \emph{symmetric} if for each of its points $z$ it has an involutive holomorphic automorphism $s_z$having $z$ as an isolated fixed point.  We consider bounded symmetric domains $\mathcal D$ in what is known as their standard Harish-Chandra realization. The \emph{irreducible} ones among these (i.e. those that are not product domains ) are in one to one correspondence with simple real Lie algebras $\mathfrak g$ such that in the Cartan decomposition $\mathfrak g= \mathfrak k + \mathfrak p$ the subalgebra $\mathfrak k$ has non-zero center. The simply connected group $\tilde{G}$ with Lie algebra $\mathfrak g$ acts on $\mathcal D$ by holomorphic automorphisms; one has $\mathcal D\cong \tilde{G}/\tilde{K}$ with $\tilde{K}$ corresponding to $\mathfrak k$.  The complexification $\mathfrak g^\mathbb C$ of $\mathfrak g$ has a vector space direct sum decomposition $\mathfrak g^\mathbb C = \mathfrak p^+ +\mathfrak k^\mathbb C + \mathfrak p^-$.  In the realization $\mathcal D$ appears as a balanced convex domain in $\mathfrak p^+\cong \mathbb C^n$. 
 
By a \emph{homogeneous} holomorphic vector bundle (\emph{hhvb}) we mean the ones homogeneous under $\tilde{G}$.  These bundles arise by the process of holomorphic induction from finite dimensional representations $(\varrho, V)$ of $\mathfrak k^\mathbb C + \mathfrak p^-$, which is a subalgebra of $\mathfrak k^\mathbb C$. The Hermitian hhvb-s (meaning homogeneous as Hermitian bundles) come from $(\varrho,V)$ such that $V$ has a $\tilde{K}$ invariant inner product. For many questions, only the existence of a Hermitian structure matters, so we will also talk about \emph{Hermitizable} hhvb-s, which can then have many Hermitian structures. 

By general principles, every holomorphic vector bundle over a domain is trivial. So, a hhvb is the same thing as a multiplier representation of $\tilde{G}$ on the space of $\Hol(\mathcal D, V)$ of $V$- valued holomorphic functions. We will keep using a certain natural trivialization which we call the \emph{canonical trivialization} (cf. \eqref{1.13-}, \eqref{1.13}).

Hermitian hhvb-s jumped into prominence in 1956, when Harish-Chandra used Hilbert spaces of sections of such bundles to construct the holomorphic discrete series of unitary representations of $\tilde{G}$. In the next three decades, the full scope of this method of constructing unitary representations was explored. All this work was about hhvb-s that are induced by irreducible representations $\varrho$ of $\mathfrak k^\mathbb C + \mathfrak p^-$ (which implies $\varrho$ is $0$ on $\mathfrak p^-$). In fact, it was clear that more general $\varrho$ can only give direct sums of representations already constructed. 

Still, the highly non-trivial more general representations of $\mathfrak k^\mathbb C + \mathfrak p^-$ and the corresponding hhvb-s exist and deserve being studied both for their own sake and for the sake of applications such as theory of Cowen-Douglas operators.  The general $(\varrho, V)$ has a descending chain of  invariant subspaces and the induced hhvb has a chain of homogeneous sub-bundles forming a composition series whose quotients are irreducible representations of $\mathfrak k^\mathbb C + \mathfrak p^-$, resp. hhvb-s induced by these. 

The first half (sections 1 and 2) of this article is devoted to this study.  The main result is Theorem \ref{thm 2.5n} which (except at some singular values of a parameter) gives an explicit differential operator $\Gamma$  (which first appeared, in the one variable case in \cite{KM0}) that intertwines in a $\tilde{G}$- equivariant way a general Hermitian hhvb with the direct sum of the factor bundles of its composition series. The prinicpal elements of the proof are Lemma \ref{lem 1.6}, which is essentially an expression for the derivative of the Jacobian matrix of a holomorphic automorphism and Theorem \ref{2.4} which is a less complicated special case of the final Theorem \ref{thm 2.5n}. 

In Section 3, we first discuss whether $\tilde{G}$- invariant Hilbert spaces, dense in $\Hol(\mathcal D, V)$, exist for our bundles. We show that this question can be completely reduced to the case of bundles induced by irreducible $(\varrho, V)$, where the answer is well-known. Then we investigate whether the gradient type operators making up $\Gamma$ in Section 2 are bounded as operators from one Hilbert space to another. We can reduce this question to the case of line bundles, but this leads to a completely satisfactory only when $\mathcal D$ is the Euclidean ball in $\mathbb C^n$. 

In Section 4, we consider homogeneous Cowen-Douglas operator $n$- tuples associated to bounded symmetric domains $\mathcal D$. For the unit disc in $\mathbb C$ there is a complete description of these in \cite{KM}. Here we extend the two main results of \cite{KM} to the case of the unit ball in $\mathbb C^n$ ($n\geq 1$); these are our Corollary \ref{rem:4.3} and Theorem \ref{thm:4.2}. Whether these results hold for more general $\mathcal D$ remains unanswered. 

The results of this article were announced previously in \cite{KM1}.

\section{Homogeneous Holomorphic vector bundles}
We consider symmetric domains $\mathcal D$ in their standard realization. We assume throughout that $\mathcal D$ is irreducible; this is sufficient for our purpose since every  bounded symmetric domain is biholomorphically equivalent to a product of such.  As Harish-Chandra showed (cf. \cite{Helg}), every irreducible $\mathcal D$ can be constructed as follows. 

Let $\mathfrak g$ be a simple non-compact real Lie algebra with Cartan decomposition $\mathfrak g = \mathfrak k + \mathfrak p$ such that $\mathfrak k$ is not semi-simple. Then $\mathfrak k$ is the direct sum of its center and  of its semisimple part, $\mathfrak k = \mathfrak z + \mathfrak k_{\rm ss},$ and there is an element $\hat{z}$ which generates $\mathfrak z$ and $\ad(\hat{z})$ is a complex structure on $\mathfrak p.$

The complexification $\mathfrak g^{\mathbb C}$ is then the direct sum $\mathfrak p^+ + \mathfrak k^\mathbb C + \mathfrak p^-$ of the $i,0,-i$ eigenspaces of $\ad(\hat{z})$.  On $\mathfrak g^{\mathbb C}$, we have the usual inner product $B_\nu(X,Y) = - B(X, \nu Y)$, where $B$ is the Killing form and $\nu$ is the conjugation with respect to the compact real form $\mathfrak k+ i \mathfrak p$. 
We let $G^\mathbb C$ denote the simply connected Lie group with Lie algebra $\mathfrak g^\mathbb C$ and we let $G, K^\mathbb C, K, P^{\pm}, Z,\ldots $ be the analytic subgroups corresponding to $\mathfrak g, \mathfrak k^\mathbb C, \mathfrak k, \mathfrak p^{\pm},\mathfrak z\ldots.$ 
We denote by $\tilde{G}$ the universal covering group of the group $G$ and by $\tilde{K},\tilde{K}_{\rm ss},\tilde{Z}\ldots $ its analytic subgroups corresponding to $\mathfrak k,\mathfrak k_{\rm ss},\mathfrak z\ldots .$  Then $\tilde{K}$ is the universal cover of $K.$ $\tilde{K}$  is also contained in  $\tilde{K}^\mathbb C,$  the universal cover of  $K^\mathbb C.$

$K^\mathbb CP^-$ is a parabolic subgroup of $G^\mathbb C.$ $P^+K^\mathbb CP^-$ is open dense in $G^\mathbb C.$ The corresponding decomposition $g^+g^0g^-$ of any $g$ in $P^+K^\mathbb CP^-$ is unique and holomorphic. The natural map $G/K \to G^\mathbb C/K^\mathbb C P^-$ is a holomorphic imbedding, its image is in the orbit of $P^+.$   Applying now $\exp_{\mathfrak p^+}^{-1}$ we get the Harish-Chandra realization of $G/K$ as a bounded symmetric domain $\mathcal D\subset\mathfrak p^+\cong \mathbb C^n.$  The kernel of the action is the (finite) center of $G$. The action of $g\in G$ on $z\in \mathcal D,$ written $g\cdot z,$  is then defined by $\exp(g\cdot z) = (g \exp z)^+.$  
We will use the notations $k(g,z) = (g \exp z)^0$ and $\exp Y(g,z) = (g \exp z)^-,$ so we have  
\begin{equation}\label{1.1}
g\exp z = (\exp(g\cdot z)) k(g,z) \exp (Y(g,z)).
\end{equation}

%
The $\tilde{G}$ - homogeneous  holomorphic vector bundles (hhvb-s) over $\mathcal D$  are 
obtained by holomorphic induction from finite dimensional joint representations of the pair $(\tilde{K}, \mathfrak k^\mathbb C+\mathfrak p^-)$. Now $\tilde{K}$ is simply connected, so this is the same as a pair $(\varrho^0, \varrho^-)$ of representations of $\mathfrak k^\mathbb C$ resp. $\mathfrak p^-$ on  a vector space $V,$  satisfying 

\begin{equation} \label{1}
\varrho^-([Z,Y]) = [\varrho^0(Z), \varrho^-(Y)], \:\:\:\: Z\in \mathfrak k^\mathbb C,\, Y\in \mathfrak p^-.
\end{equation}
This condition can also be equivalently written as 
\begin{equation} \label{1.1equiv} 
\varrho^- (\Ad(k) Y) = \varrho^0(k) \varrho^-(Y)\varrho^0(k)^{-1}, \,\, Y\in \mathfrak p^-, k\in \tilde{K}.
\end{equation}
(We use the same symbols to denote the representations of Lie groups and their Lie algebras.)

We will refer to such a pair simply as the representation $(\varrho, V)$.

The homogeneous Hermitian  holomorphic vector bundles 
arise from representations $(\varrho, V)$ such that $V$ has an (arbitrary, fixed) $\varrho^0(\tilde{K})$- invariant inner product. In this case, we call $(\varrho, V)$ a Hermitian representation. 

For an invariant inner product to exist on $V$ it is necessary and sufficient (since $\hat{z}$ generates $\mathfrak z$ and $\tilde{K}_{\rm ss}$ is compact) that $\varrho^0(\hat{z})$ should be diagonalizable and have purely imaginary eigenvalues.  If $(\varrho, V)$ has this property, we say it is a Hermitizable representation, and the holomorphically induced bundle is a Hermitizable  homogeneous holomorphic vector bundle (abbreviated \emph{Hhhvb}). 

Given a Hhhvb, it is easy to describe all its possible structures making it homogeneously Hermitian, and most of our results will be independent of the particular structure chosen.  This is mainly due to the following well-known consequence of Schur's Lemma: In the direct decomposition of $V$ under $\varrho^0(\tilde{K})$, the isotypic subspaces are orthogonal to each other, no matter which invariant inner product is chosen.  Such representations and bundles are the main  objects of our study. 

%
%
%

Since $\hat{z}$ spans the center $\mathfrak z$ of $\mathfrak k,$ $\chi_\lambda(\hat{z})=i\lambda$ defines a character of $\mathfrak k.$
By Schur's Lemma $V$ is the orthogonal sum of $\varrho^0$- invariant subspaces $V^\lambda$ on which $\varrho^0(\hat{z}) = i \lambda$ $(\lambda\in \mathbb R).$  For any $\lambda,$  we have 
\begin{equation} \label{1.2}  
\varrho^-(Y)V^\lambda \subseteq V^{\lambda-1}, \,\, Y\in \mathfrak p^-
\end{equation}
because for any $v_\lambda$ in $V^\lambda,$
\begin{eqnarray*}
\varrho^\circ(\exp t\hat{z}) (\varrho^-(Y) v_\lambda) &=& \varrho^-(\mathrm{Ad}(\exp\, t \hat{z})Y) \varrho^\circ(\exp\, t \hat{z}) 
v_\lambda\\
&=&\varrho^-(e^{-it} Y) e^{\lambda i t} v_\lambda\\
&=& e^{(\lambda-1)it}\big (\varrho^-(Y) v_\lambda\big ). 
\end{eqnarray*}
It follows immediately that for indecomposable Hermitizable $(\varrho,V)$ we have 
\begin{equation} \label{dirdecV_j}
V=\oplus_{j=0}^m V_j
\end{equation}
an orthogonal sum of representations $(\varrho_j^0, V_j)$ of $\mathfrak{k}$
such that $\varrho_j^0(\hat{z})=i(\lambda-j)$ with some $\lambda \in \mathbb R$
determined by $\varrho.$ Writing, for $Y\in \mathfrak{p}^-, \,1\leq j \leq m,$ 
$$
\varrho_j^-(Y)= \varrho^-(Y)_{|\,_{V{j-1}}}
$$
we have $\varrho_j^-(Y)\in \Hom(V_{j-1},V_j).$ These will be our standing notations. We observe that $\varrho^-(Y)$ is just the direct sum of the $\varrho_j^-(Y),$ $(1\leq j \leq m).$ We also note that \eqref{1.1equiv} can be written in the concise form 
\begin{equation}\label{1.1concise}
\varrho_j^- \in \Hom \big(\mathfrak p^-, \Hom(V_{j-1},V_j)\big )^{\tilde{K}},
\end{equation}
where the superscript $\tilde{K}$ means the $\tilde{K}$- invariant elements in the space.

At this point, we note that an indecomposable $\varrho$ determines a real number $\lambda.$ So we can always write $\varrho = \chi_\lambda \otimes \varrho^{\rm nor},$ where the real number determined by $\varrho^{\rm nor}$ is $0.$

Setting 
$$\tilde{V}_j=V_j\oplus \cdots \oplus V_m$$
it is clear that $\tilde{V}_j$ is an invariant subspace for $\varrho$.  The representation induced by $\varrho$ on $\tilde{V}_j/\tilde{V}_{j+1}$ is isomorphic with the representation $(\varrho^0_j,0)$ (meaning $\varrho^0_j$ on $\mathfrak k^\mathbb C$ and $0$ on $\mathfrak p^-$.)

We write $\Ad_{\mathfrak p^+}, \Ad_{\mathfrak p^-}$ for the adjoint representation restricted to $\tilde{K}$ or $\tilde{K}^\mathbb C$ or $\mathfrak k^\mathbb C$ acting on $\mathfrak p^+,$ resp. $\mathfrak p^-.$ They are irreducible (since $\mathfrak g$ is simple) and they leave invariant the natural Hermitian inner product $B_\nu$ of $\mathfrak g^\mathbb C$ restricted to $\mathfrak p^\pm.$

\begin{lem} \label{lem 1.1n}
Let $W \subseteq V_{j-1}$ be an irreducible subspace for $\varrho^0_{j-1}.$ Then the subspace $\varrho_j^-(\mathfrak p^-)W$ of $V_j$ as a $\tilde{K}$ - representation is equivalent to a subrepresentation
of $\mathfrak p^-\otimes W.$
\end{lem}
\begin{proof}
The map $(Y,w) \mapsto \varrho_j^-(Y)w$ of $\mathfrak p^-\times W$ (hence also of $\mathfrak p^-\otimes W$) into $V_j$ is $\tilde{K}$-  equivariant, since we have 
$$
\varrho_j^-(\Ad(k)Y) \varrho_{j-1}^0(k)w = \varrho_j^0(k)\varrho_j^-(Y)w
$$ 
by \eqref{1.1equiv}. The range of the map is then isomorphic to the $\tilde{K}$ invariant complement of its kernel. 
\end{proof}

The analysis of the Hermitian representations of $\mathfrak k^\mathbb C +\mathfrak p^-$ will be continued in Section 2. Here, the following lemmas lead to Proposition \ref{prop 1.1} and Theorem \ref{prop 1.2}, a structure theorem of indecomposable Hhhvb-s, which we will actually not use in the rest of this paper. 

As usual, for an irreducible $\mathcal D$, we write $p= (r-1)a + b +2$, where $r$ is the rank and $a, b$ are the multiplicities of the long 
(resp. short) restricted roots different from the Harish-Chandra strongly orthogonal roots.  

\begin{lem} \label{lem 1.1}
$Z\cap K_{\rm ss}$ is a finite cyclic group generated by $\exp t \hat{z},$ where $t= 2\pi \frac{p}{n}$ is the smallest positive $t$ such that $\exp t \hat{z}\in K_{\rm ss}.$
\end{lem} 
\begin{proof}
The group is finite since it is central in $K_{\rm ss}$ and cyclic because it is a subgroup of $Z$. Using the computation in \cite[Sec. 3]{KW} of the relation between $\hat{z}$ and the generator used by Schlichtkrull, \cite[Prop 3.4]{Schl} gives that $\exp \,t \frac{p}{n} \hat{z}\in K_{\rm ss}$ if and only if $t \in 2 \pi \mathbb Z$.  This implies the Lemma.
\end{proof}

We write $\pi: \tilde{G} \to G$ for the covering map.  Corresponding to the direct product $\tilde{K} = \tilde{Z} \cdot \tilde{K}_{\rm ss}$, every irreducible representation of $\tilde{K}$ is uniquely of the form $\chi_\lambda \otimes \sigma$, where 
$\chi_\lambda(\exp t \hat{z}) = e^{i t \lambda}$ and $\sigma$ is an irreducible representation of $\tilde{K}_{\rm ss}$ extended trivially to  $Z.$  By \cite[Cor. 3.2]{Schl}, $K_{\rm ss}$ is simply connected, so $\pi_{\vert \tilde{K}_{\rm ss}}$ is an isomorphism.  By Lemma \ref{lem 1.1}, $\big ( \pi_{| \tilde{K}_{\rm ss}}\big )^{-1}(\exp_G(2 \pi \frac{p}{n} \hat{z}))$ is in the center of $\tilde{K}_{\rm ss},$ hence of $\tilde{K}.$  So, by Schur's Lemma, there is a well-defined residue class $\Lambda(\sigma)$ in $\mathbb R/\frac{n}{p} \mathbb Z$ such that 
\begin{equation} \sigma \big ( (\pi_{| \tilde{K}_{\rm ss}})^{-1}(\exp_G 2\pi \frac{p}{n}\hat{z}) \big ) = e^{2 \pi i \frac{p}{n} \Lambda(\sigma)} I \end{equation}
(with a little abuse of notation).

We write  ${\rm Ad}^\prime_{\mathfrak p^-}$ for  ${\rm Ad}_{\mathfrak p^-}$ restricted to $\tilde{K}_{\rm ss}.$  By ${\rm ad}_{\mathfrak p^-}( \hat{z})= -i,$ we have, $\Ad_{\mathfrak p^-}= \chi_{-1} \otimes\Ad_{\mathfrak p^-}^\prime,$ and so 
\begin{equation}
\Lambda({\rm Ad}^\prime_{\mathfrak p^-})\equiv -1
\end{equation}
\begin{lem} \label{lem 1.2}
\begin{enumerate} \item[(i)] If $\sigma,\, \sigma^\prime,\, \sigma^{\prime\prime}$ are irreducible and $\sigma$ is contained in $\sigma^\prime\otimes \sigma^{\prime\prime}$, then $\Lambda(\sigma) = \Lambda(\sigma^\prime) + \Lambda(\sigma^{\prime \prime})$.   
\item[(ii)] The irreducible representation $\chi_\lambda\otimes \sigma$ of $\tilde{K}$ is the lift under $\pi$ of a representation of $K$ if and only if $\lambda \in \Lambda(\sigma)$.  
\end{enumerate}
\end{lem}
\begin{proof}
(i) is trivial.  For (ii), we note that $\pi$ maps a generic element $\tilde{k}_{\rm ss}\exp_{\tilde{G}}t \hat{z}$ to $\big ( \pi_{|\tilde{K}_{\rm ss}}\big ) (\tilde{k}_{\rm ss}) \exp_G t\hat{z}$.  $\chi_\lambda\otimes \sigma$ is a lift if and only if it is trivial on $\ker(\pi_{| \tilde{K}})$, i.e., if  and only if the condition 
\begin{equation} \label{1.4} \pi(\tilde{k}_{\rm ss}) = \exp_G - t \hat{z}\end{equation}
implies $\sigma(\tilde{k}_{\rm ss}) = e^{-i \lambda t}I$.  By Lemma \ref{lem 1.1}, \eqref{1.4} holds for some $\tilde{k}_{\rm ss}$ if and only if $t=- 2 \pi \frac{p}{n} \ell$ with $\ell \in \mathbb Z$, and in this case $\tilde{k}_{\rm ss}= \big ( \pi_{|\tilde{K}_{\rm ss}}\big )^{-1}(\exp_G 2 \pi \frac{p}{n} \ell)$.  So, finally $\chi_\lambda \otimes \sigma$ is a lift if and only if
\begin{equation} \label{1.5}\sigma \big ( \big ( \pi_{|\tilde{K}_{\rm ss}}\big )^{-1}(\exp_G 2 \pi \frac{p}{n} \ell \hat{z})\big ) = e^{2 \pi i\frac{p}{n} \ell \lambda} I
\end{equation}
for all $\ell \in \mathbb Z$. Clearly, this holds for all $\ell$ if and only if it holds for $\ell =1$.  For $\ell=1$, the left hand side is $e^{2 \pi i\frac{p}{n}  \Lambda(\sigma)}I$ by definition of $\Lambda(\sigma)$.  Hence \eqref{1.5}  holds if and only if $\lambda \in \Lambda(\sigma)$ finishing the proof.

\end{proof}



\begin{defn}
A Hermitizable  representation $(\varrho, V)$ of $\mathfrak k^\mathbb C+\mathfrak p^-$ (and the Hhhvb induced by it) is said to be \emph{elementary} if for some $\lambda \in \mathbb R$, $m\in \mathbb N$, it is of the form $\oplus_{j=0}^m V_j$ with $\ad(\hat{z})=i(\lambda-j)$ on $V_j$ and if $\Lambda(\sigma) + j$ is the same for every irreducible component $\sigma$ of $\varrho_{j}^{0,{\rm ss}}$ (meaning $\varrho^0_{j}$ restriccted to $\mathfrak k_{\rm ss}),$ for every $0 \leq j \leq m$.
\end{defn}
\begin{prop} \label{prop 1.1}
If $(\varrho, V)$ is an  indecomposable Hermitian representation of $\mathfrak k^\mathbb C+ \mathfrak p^-$ (i.e.  the induced holomorphic homogeneous Hermitian vector bundle is irreducible), then it is elementary.
\end{prop}

\begin{proof} Let $(\varrho, V)$be indecomposable. In \eqref{dirdecV_j} we have already seen that there is a decomposition $\oplus V_j$ as stated. Now, let $\Lambda$ be a residue class in $\mathbb R/\tfrac{n}{p} \mathbb Z$ and let $V(\Lambda)$ denote the direct sum of all the irreducible constituents $\sigma$ of $V_j,$ for each $1 \leq j \leq m,$   such that $\Lambda(\sigma)+j=\Lambda.$  It is immediate from Lemmas \ref{lem 1.1n} and \ref{lem 1.2} that $V(\Lambda)$ is invariant under both $\varrho^0$ and $\varrho^-.$ Hence by indecomposability, there can be only one class $\Lambda$  such that $V(\Lambda) \not=0.$

%
%
\end{proof}
\begin{thm} \label{prop 1.2}
Every elementary  { Hhhvb} $E$ can be written as a tensor product $L_{\lambda_0} \otimes E^\prime$, where $L_{\lambda_0}$ is the line bundle induced by the character $\chi_{\lambda_0}$ and $E$ is the lift to $\tilde{G}$ of a $G$ - homogeneous  holomorphic  Hermitian vector bundle which is the restriction to $G$ and $\mathcal D$ of a $G^\mathbb C$ - homogeneous vector bundle over $G^\mathbb C/K^\mathbb C P^-$ induced in the holomorphic category by a representation of $K^\mathbb CP^-$.
\end{thm}
\begin{proof} Suppose $E$ is induced by $(\varrho, V)$, $V=\oplus_0^m V^{\lambda -j}$.  We can take any irreducible component $\sigma$ of $\varrho_\lambda^{0,{\rm ss}},$ choose some $\lambda^\prime \in \Lambda(\sigma)$ and set $\lambda_0=\lambda-\lambda^\prime.$  Then we can write $\varrho^0=\chi_{\lambda_0}\otimes \varrho^{\prime,0}$. 
By Lemma \ref{lem 1.2}, $\varrho_\lambda^{\prime,0}$ is a lift of a representation of $K$ to $\tilde{K}$ and also a representation of $\mathfrak k^\mathbb C$. It follows that it extends then to a holomorphic representation of ${K}^\mathbb C$. The  $\varrho^-$ part which is unchanged gives a representation of $P^-$ since $P^-$ is simply connected.  So, we have a representation of the semidirect product $\tilde{K}^\mathbb C P^-$, and the Theorem follows.       
\end{proof}



We will study our irreducible \emph{Hhhvb}-s through a natural trivialization which can be obtained in one of  two ways. One way is based on Theorem \ref{prop 1.2}, putting together the natural trivializations of $L_\lambda$ (where the multiplier is a power of the jacobian, see e.g. \cite{KW}) and a trivialization of $E^\prime$ built from $k(g, z)$ (as defined in \eqref{1.1}). The other way, which we will actually follow, makes use of the Herb-Wolf local complexification of $\tilde{G}$ (cf. \cite{HW}). In either approach, the point is to define  a $\tilde{K}^\mathbb C$-valued multiplier $\tilde{k}(g, z)$ and prove its properties.  

We write $\pi:\tilde{K}^\mathbb C \to K^\mathbb C$ for the universal covering map. As shown in \cite{HW}, $P^+ \times \tilde{K}^\mathbb C \times P^-$ can be given 
a structure of complex analytic local group such that (writing $\pi: \tilde{K}^\mathbb C \to K^\mathbb C$)
${\rm id} \times \pi \times{\rm id}$ is the universal local group covering of $P^+K^\mathbb C P^-.$
We write $\tilde{G}_{\rm loc}$ for this local group and abbreviate ${\rm id} \times \pi \times{\rm id}$ to $\pi.$  
By \cite{HW}, $\tilde{G},\, \tilde{K}^\mathbb C P^-,$ $P^+ \tilde{K}^\mathbb C$ are closed subgroups of 
$\tilde{G}^\mathbb C_{\rm loc}$ and $\tilde{G} \exp \mathcal D \subset \tilde{G}^\mathbb C_{\rm loc}.$ $\pi$ restricted to $\tilde{G}$ is the covering map of $G.$
Defining $g\cdot z = \pi(g) \cdot z$ and $Y(g,z) = Y(\pi(g),z)$ we have the decomposition
\begin{equation} \label{1.11}
g \exp z = (\exp g\cdot z) \tilde{k}(g,z)\exp Y(g,z), \,\, (g\in \tilde{G}, z\in\mathcal D)
\end{equation}
in $\tilde{G}_{\rm loc}.$ We write $\tilde{b}(g,z) = \tilde{k}(g,z) \exp Y(g,z).$ Then  applying \eqref{1.11} twice, we have 
$$
(\exp g g^\prime z) \tilde{b}(g g^\prime,z) = g g^\prime \exp z = g(\exp g^\prime z) \tilde{b}(g^\prime, z)=(\exp g g^\prime z)\tilde{b}(g,g^\prime z) \tilde{b}(g^\prime,z)
$$
which shows that $\tilde{b}(g,z)$ satisfies the multiplier identity 
\begin{equation} \label{tildeb}
\tilde{b}(gg^\prime,z) = \tilde{b}(g,g^\prime z) \tilde{b}(g^\prime,z).
\end{equation}
Furthermore, we clearly have $\tilde{b}(kp^-,0) = k p^-$ for $kp^- \in \tilde{K}^\mathbb C P^-.$  

It follows that given a representation  $(\varrho,V)$ of $\mathfrak k^\mathbb C + \mathfrak p^-$ as above, $\varrho(\tilde{b}(g,z))$ is a multiplier, and 
\begin{equation}\label{1.13-}
\varrho(\tilde{b}(g,z)) = \varrho^0(\tilde{k}(g,z))\varrho^-(\exp Y(g,z)).
\end{equation}
The vector bundle $E^\varrho$ holomorphically induced by $\varrho$  has a  trivialization to be called the \emph{canonical trivialization} in which the space of sections is $\Hol(\mathcal D,V),$  and the $\tilde{G}$ - action on it is the multiplier representation $U^\varrho:$ 
\begin{equation}\label{1.13}
\big (U^\varrho_g f \big ) (z) = \varrho(\tilde{b}(g^{-1},z))^{-1} f(g^{-1} z).
\end{equation} 
The canonical trivialization will be used throughout the rest of this paper.

It is clear from the product expression \eqref{1.13-} that $\Hol(\mathcal D, \tilde{V}_j)$ for each $j$, is an $U^\varrho$-invariant subspace of $\Hol(\mathcal D, V)$, and the representation induced by $U^\varrho$ on $\Hol(\mathcal D, \tilde{V}_j)/\Hol(\mathcal D, \tilde{V}_{j+1}$ is the same as the representation on $\Hol(\mathcal D,V_j)$  via the multiplier $\varrho^0_j(\tilde{k}(g,z))$. In other words, we have a chain of homogeneous sub-bundles 
$\tilde{E}_j$ with $E_j = \tilde{E}_j /\tilde{E}_{j+1}$ holomorphically induced by $(\varrho^0_j,0)$ on $V_j$.

If $f\in\Hol(\mathcal D,V),$ then we write $Df$ for the derivative: $Df(z) X = (D_Xf)(z)$ for $X\in \mathfrak p^+.$ 
Thus $Df(z)$ is a $\mathbb C$ - linear map from $\mathfrak p^+$ to $V.$ The following Lemma is crucial for the computations of Section 2. 
\begin{lem} \label{lem 1.6}
For any holomorphic representation $\tau$ of $\tilde{K}^\mathbb C$ and any $g\in \tilde{G}$, $z\in\mathcal D$, $X\in \mathfrak p^+$,  
$$D_X\tau\big (\tilde{k}(g,z)^{-1}\big ) = - \tau\big ([ Y(g,z), X]\big )\tau\big (\tilde{k}(g,z)^{-1}\big ).$$
Furthermore,
$$
D_X\,Y(g,z)=\frac{1}{2} \big [Y(g,z),[Y(g,z),X]\big ].
$$
\end{lem}
\begin{proof} We have, using the exponential map of $\tilde{G}^\mathbb C_{\rm loc}$
\begin{eqnarray*}
g\exp(z+tX) &=& g \exp z\,\, \exp tX\\  
&=&\exp( g z) \tilde{k}(g,z) \exp Y \exp tX \\
&=& \exp(g z) \tilde{k}(g,z) \exp\,t\{X+[Y,X] +\frac{1}{2} [Y,[Y,X]]\} \exp (Y),
\end{eqnarray*}
where we have used the abbreviated notation $Y=Y(g,z)$.  By the Campbell-Hausdorff formula, 
and since $\tilde{K}^\mathbb C$ normalizes $P^+,$ equating the $\tilde{K}^\mathbb C$ parts of the two sides, we get
$$
\tilde{k}(g,z+tX) = \tilde{k}(g,z)\,\exp(t[Y,X]+O(t^2)).
$$
Applying $\tau$ to the inverse and taking $\frac{d}{dt}_{{\mid 0}}$ gives the first statement. Looking at the $P^-$ part of the decomposition
we get the second statement. 
\end{proof}
\begin{rem}\label{rem:1.8}
Equating the $P^+$- parts of the identity above we get 
$$
\exp g(z+tX) = \exp (gz + t \Ad(\tilde{k}(g,z)) X + O(t^2)),
$$
whence, slightly extending \cite[p. 65]{S}, for all $g\in \tilde{G},$ we have 
\begin{equation} \label{1.14}
Dg(z) = \Ad_{\mathfrak p^+}\tilde{k}(g,z).
\end{equation}
Further we note that by the general identity $DF=-F(DF^{-1})F$ we also know $D\tau(\tilde{k}(g,z)).$    Taking $\tau = \Ad_{\mathfrak p^+},$ and using \eqref{1.14},  the Lemma also gives an explicit expression for $D^2g(z).$
\end{rem}

\section{The main results about vector bundles}
For a more detailed description of the indecomposable Hermitizable representations $(\varrho, V)$ of $\mathfrak k^\mathbb C + \mathfrak p^-$ we have to make some normalizations. We already know that $\varrho$ determines a real number $\lambda;$ we also keep using the decomposition $V_0 \oplus \cdots \oplus V_m$ and the restrictions $\varrho^0,\,\varrho^0_i\, (1\leq i \leq m),\, \varrho^-$ of $\varrho$ as in Section 1.  We consider the set of all irreducible representations $(\alpha, W^\alpha)$ of $\mathfrak k_{\rm ss}^\mathbb C$ and choose a fixed $\tilde{K}$- invariant inner product $\langle \cdot , \cdot \rangle_\alpha$ on $W^\alpha;$ when $\alpha = \Ad^\prime_{\mathfrak p^-}$ we choose the restriction of $B_\nu$ to to $\mathfrak p^-.$ 


For any $\alpha$ the tensor product $\Ad^\prime_{\mathfrak p^-}\otimes \alpha$ is multiplicity free (cf. \cite[Corollary 4.4]{Jak}, or in a wider context \cite{JKR}). for every irreducible component $\beta$ of it we choose and fix an equivariant partial isometry $P_{\alpha \beta}: \mathfrak p^- \otimes W^\alpha \to W^\beta.$ For $Y\in \mathfrak p^-$ and $v\in W^\alpha,$ we define 
\begin{equation}
\tilde{\varrho}_{\alpha \beta} (Y) v = P_{\alpha \beta} (Y\otimes v).
\end{equation}
%

We start our closer study of the indecomposable Hermitizable representations $(\varrho, V)$ and the corresponding Hhhvb with the case where $\varrho$ is irreducible. Then $m=0$ (since each $\tilde{V}_j=V_j+ \cdots + V_m$ is always an invariant subspace). This implies $\varrho^-=0$ and hence $\varrho^0$ is irreducible, i.e., $\varrho^0 = \chi_\lambda \otimes \alpha$ with some $\alpha$ as above.  We may assume, without restriction of generality, that $V=W^\alpha$ as a vector space; the possible inner products are $\langle\cdot, \cdot \rangle_V = H \langle\cdot, \cdot \rangle_\alpha$ with some number $H> 0.$ We denote the corresponding Hhhvb by $E^{\alpha, \lambda}$. 

A little more generally, when we have a multiple of an irreducible $\varrho,$ i.e. $\varrho^-=0$ and $\varrho^0=\chi_\lambda I_\alpha \otimes \alpha$ on $V=\mathbb C^d \otimes W^\alpha$, the $\tilde{K}$ invariant inner products on $V$ are the tensor products of $\langle H \cdot, \cdot \rangle_{\mathbb C^d}$ on $\mathbb C^d$ and $\langle\cdot, \cdot \rangle_\alpha$ on $W^\alpha$ with some positive definite matrix $H$. 

(We might note that at this point all choices of $\mu$ still give isometrically isomorphic Hermitian representations and homogeneous holomorphic vector bundles) 

We will now study the case of indecomposable $(\varrho,V)$ such that $m$ is arbitrary and each summand $(\varrho_j^0,V_j)$ is irreducible. We say that such a $\varrho$ and the corresponding Hhhvb are \emph{filiform}. This case is the key to the general case. 

Now, $\varrho_j^0 = \chi_{\lambda -j}\otimes \alpha_j\,\, (0\leq j \leq m)$ and $V_j$ is $W^{\alpha_j}$ as a vector space with eventual inner product determined by a positive number $H_j$. We will use the abbreviations $W_j = W^{\alpha_j}, \, P_j = P_{\alpha_{j-1}, \alpha_j}, \, \tilde{\varrho}_j = \tilde{\varrho}_{\alpha_{j-1}, \alpha_j}$. So

\begin{equation} \label{2}
\tilde{\varrho}_j(Y)v=P_j(Y\otimes v)
\end{equation}
for $Y\in \mathfrak p^-,\,v\in V_{j-1}$. 
Clearly $\tilde{\varrho}_j$ satisfies \eqref{1.1concise}. This space of $\tilde{K}$ invariants is isomorphic with the space of $\tilde{K}$- equivariant maps $\mathfrak p^-\otimes V_{j-1} \to V_j,$  hence is $1$ dimensional. 
It follows that 
$$\varrho_j^-(Y) = y_j \tilde{\varrho}_j(Y)$$ 
for each $j$ with a number $y_j$. Also, $y_j\not = 0$ since indecomposability is part of the definition of filiform. Since the $\varrho_j^-\,\,(1 \leq j \leq m)$ 
together form  a representation of the Abelian Lie algebra $\mathfrak p^-$, we have   
\begin{equation}\label{2.2}
\tilde{\varrho}_{j+1}(Y^\prime) \tilde{\varrho}_j(Y) = \tilde{\varrho}_{j+1}(Y)\tilde{\varrho}_j(Y^\prime)
\end{equation}
for all $Y,Y^\prime\in \mathfrak p^-$ and $1\leq j \leq m-1.$
In terms of $P_j,$ this means
\begin{equation}\label{3}
P_{j+1}\big (Y^\prime \otimes P_j(Y\otimes v) \big )= P_{j+1}\big (Y \otimes P_j(Y^\prime\otimes v) \big )
\end{equation}
for all $Y,Y^\prime \in \mathfrak p^-$ and $v\in V_{j-1}.$
A third equivalent way to write this  condition is 
\begin{equation}\label{2.4new}
P_{j+1} \tilde{\varrho}_j(Y) = \tilde{\varrho}_{j+1}(Y) P_j.
\end{equation}
(Here on the left hand side $\tilde{\varrho}_{j}(Y)$ is really an abbreviation for $I\otimes\tilde{\varrho}_j(Y).$)

To summarize,  any sequence $\boldsymbol \alpha= (\alpha_0, \ldots , \alpha_m)$ such that $\alpha_j$ is contained in $\Ad_{\mathfrak p^-}\otimes\alpha_{j-1}$ and such that the $P_J$-s satisfy \eqref{3}, together with $\lambda \in \mathbb R$ and a sequence $y=(y_1, \ldots , y_n)$ of non-zero numbers determine a filiform Hermitizable representation.  Its possible Hermitian structures are given by sequences $H=(H_0, \ldots , H_m)$ of positive numbers.  There is  considerable redundancy here. (In fact, all choices of $y_j\not= 0 \,\, (\forall j)$  give isomorphic hhvb-s, while $y_j >0,\, H_j=1 \, (\forall j)$ is one possible normalization of the hhvb-s.) 
But this is not important at this point.  

We proceed towards Theorem \ref{2.4}, the main result about the filiform case. 
    

We denote by $\iota$ the identification of $(\mathfrak p^+)^*$ with $\mathfrak p^-$ under the Killing form, and for any vector space $W,$ extend it to a map from $\Hom(\mathfrak p^+, W)$ to $\mathfrak p^-\otimes W;$ that is, for $Y\in \mathfrak p^-, w\in W,$ 
$$
\iota(B(\cdot,Y)w) = Y\otimes w.
$$

For any $T\in {\rm Hom}(\mathfrak p^+,W)$ and  for all $k\in \tilde{K}^\mathbb C,$ the invariance of $B$ implies 
\begin{equation} \label{1.6old}
\iota\big (T\circ{\rm Ad}_{\mathfrak p^+}(k^{-1})\big ) = {\rm Ad}_{\mathfrak p^-}(k) \iota(T).
\end{equation}

Of course, linear transformations affecting only $W$ commute with $\iota.$
In particular, as in our later applications, if $W$ is some space of linear transformations $F_1 \to F_2$ and $U:F_2 \to F_3$ and $V:F_0\to F_1$ are fixed linear transformations, then 
\begin{equation}\label{2.6}
\iota(UTV) = U\iota(T) V,
\end{equation}

\begin{lem} \label{lem 2.1}
Let $\varrho$ be a filiform representation. Then there exist  
constants $u,w$  independent of $\lambda,$ such that for all $Y\in \mathfrak p^-$, we have
\begin{equation}\label{2.7}
P_j \iota \varrho^0_{j-1}([Y,\cdot]) =  c_j(\lambda) \tilde{\varrho}_j(Y),
\end{equation}
where 
\begin{equation}\label{2.8}
c_j(\lambda) = (u + (j-1)w - \tfrac{\lambda}{2n}).
\end{equation}
\end{lem}

\begin{proof}
We have $\varrho^0_{j-1} = \chi_{\lambda-j+1} \otimes{\varrho_{j-1}^0}^\prime,$ with ${\varrho_{j-1}^0}^\prime$ trivial on $\mathfrak z$ (i.e. a representation of $\tilde{K}_{\rm ss}$).  Now
\begin{equation} \label{Lemma2.1(1)}
\varrho^0_{j-1}([Y,\cdot])=\chi_{\lambda-j+1}([Y,\cdot])\otimes I_{V_{j-1}} + {\varrho^0_{j-1}}^\prime([Y,\cdot]).
\end{equation}
The first term, evaluated on $X\in \mathfrak p^+,$ depends only on the projection of $[Y,X]$ onto $\mathfrak z^\mathbb C.$ This projection is equal to 
$$
\tfrac{B_\nu([Y,X],\hat{z})}{B_\nu(\hat{z},\hat{z})} \hat{z} = \tfrac{B([\hat{z}, Y], X)}{B(\hat{z},\hat{z})} \hat{z} = \tfrac{i}{2n} B(Y,X) \hat{z},
$$
where we have used  $\nu \hat{z} = \hat{z}$ and $B(\hat{z},\hat{z})=-2n.$
Hence 
$$\chi_{\lambda-j+1}([Y,\cdot]) = - \tfrac{\lambda - j + 1}{2n} B(Y,\cdot), $$ 
and so, applying $P_j\circ \iota$ to the first term on the right in  \eqref{Lemma2.1(1)} we obtain
\begin{equation}\label{2.10}
-\tfrac{\lambda-j+1}{2n} \tilde{\varrho}_j(Y).
\end{equation}
Next we apply $P_j\circ \iota$ to the second term in \eqref{Lemma2.1(1)}. We get an element of ${\rm Hom}(V_{j-1}, V_j)$ which because of \eqref{1.6old} and the equivariance of $P_j$ depends on $Y\in \mathfrak p^-$ in a $\tilde{K}$-equivariant way. But we already know that every equivariant map from $\mathfrak p^-$ to ${\rm Hom}(V_{j-1}, V_j)$ is a constant multiple of $\tilde{\varrho}_j.$ 
Putting this together with \eqref{2.10} we have \eqref{2.7} with 
\begin{equation}
c_j(\lambda) = c_j^\prime - \tfrac{\lambda}{2n},
\end{equation}
where $c_j^\prime$ is some constant independent of $\lambda.$

To prove \eqref{2.8} it will be enough to prove that $c_{j+1}(\lambda) - c_j(\lambda)$ is independent of $j,$ ($1 \leq j \leq m-1$). For this we give another expression for the left hand side of \eqref{2.7}. 
Let $\{e_\beta\}$ be a basis for $\mathfrak p^+$ and $\{e_{-\beta}^\prime\}$ the $B$-dual basis of $\mathfrak p^-.$ Expanding an arbitrary $X\in \mathfrak p^+$ in terms of the basis we have 
$$
\varrho^0_{j-1}([Y,X]) = \sum_\beta B(e_{-\beta}^\prime,X) \varrho_{j-1}^0([Y,e_\beta])
$$
and 
$$
\iota \varrho^0_{j-1}([Y,\cdot]) = \sum_\beta e_{-\beta}^\prime\otimes \varrho_{j-1}^0([Y,e_\beta])
$$
Applying $P_j$ to this, we can rewrite \eqref{2.7} as 
\begin{equation}\label{2.12}
c_j(\lambda) \tilde{\varrho}_j(Y)=\sum_\beta \tilde{\varrho}_{j}(e^\prime_{-\beta}) \varrho_{j-1}^0([Y,e_\beta])
\end{equation} 
We choose $Y, Y^\prime, Y^{\prime\!\prime}$ in $\mathfrak p^-$ such that 
\begin{equation}\label{2.13}
\tilde{\varrho}_{j+2}(Y^{\prime\!\prime}) \tilde{\varrho}_{j+1}(Y^\prime) \tilde{\varrho}_{j}(Y) \not = 0.
\end{equation} 
This is possible by the irreducibility of each $V_j.$ Now we write \eqref{2.12} with $j+1$ instead of $j,$ multiply on the right by $\tilde{\varrho}_j(Y^\prime),$ then use that $\varrho$ is a representation of $\mathfrak k^\mathbb C +\mathfrak p^-:$
\begin{align}\label{2.14}
c_{j+1}(\lambda) \tilde{\varrho}_{j+1}(Y) \tilde{\varrho}_{j}(Y^\prime) = \phantom{\quad \quad \quad \quad \quad \quad \quad Gadadhar Misra Adam Koranyi}& \nonumber \\
 \sum_\beta \tilde{\varrho}_{j+1}(e_{-\beta}^\prime) \tilde{\varrho}_{j}(Y^\prime) \varrho_{j-1}^0([Y,e_\beta]) + \sum_\beta \tilde{\varrho}_{j+1}(e_{-\beta}^\prime)
 \tilde{\varrho}_{j}([[Y,e_\beta],Y^\prime]).
\end{align}
We multiply \eqref{2.12} on the left by $\tilde{\varrho}_{j+1}(Y^\prime)$ and subtract it from \eqref{2.14}. Using \eqref{2.2} on both sides, we obtain
\begin{equation}\label{2.15}
\big ( c_{j+1}(\lambda)  - c_j(\lambda) \big ) \tilde{\varrho}_{j+1}(Y^\prime) \tilde{\varrho}_j(Y) = \sum_\beta \tilde{\varrho}_{j+1}(e^\prime_{-\beta}) \tilde{\varrho}_{j}\big ( [ [Y,e_\beta], Y^\prime ] \big ). 
\end{equation}
Now we write this with $j+1$ in place of $j,$ multiply on the right by $\tilde{\varrho}_j(Y^{\prime\!\prime}),$ and compare the resulting equality with \eqref{2.14} left multiplied by $\tilde{\varrho}_{j+2}(Y^{\prime\!\prime}).$ By \eqref{2.2}, the right hand sides are equal, and by \eqref{2.13} it follows that 
$$
c_{j+2}(\lambda) - c_{j+1}(\lambda)  = c_{j+1}(\lambda) - c_{j}(\lambda).  
$$
Since this holds for every $j,$ the proof is complete.
\end{proof}
%
%


\begin{lem} \label{lem 2.2}
Let $\varrho$ be a filiform representation.  For all $0 \leq j \leq m-1,$ and holomorphic $F:\mathcal D\to V_j$,
\begin{eqnarray*}
\lefteqn{P_{j+1} \iota D^{(z)} \big \{\varrho_j^0 (\tilde{k}(g,z)^{-1})F(g z) \,\big \}}\\
&=& -c_{j+1}(\lambda) \tilde{\varrho}_{j+1} \big (Y(g,z)\big )\big (\varrho_j^0 (\tilde{k}(g,z)^{-1})F(g z) \,\big) + \varrho_{j+1}^0 (\tilde{k}(g,z)^{-1}) \big ( (P_{j+1} \iota D F)(gz)\big ),\end{eqnarray*}
where $D^{(z)}$ denotes differentiation with respect to $z.$ 
\end{lem}
\begin{proof}
Applying the Leibniz product rule on the left hand side we get 
$$
P_{j+1} \iota \big (D^{(z)}{\varrho}^0_j(\tilde{k}(g,z)^{-1}) \big ) F(g z) + P_{j} \iota {\varrho}_j^0(\tilde{k}(g,z)^{-1}) D^{(z)} \{F(g z)\}.
$$
To the first term we apply Lemma \ref{lem 1.6}, then Lemma \ref{lem 2.1} and obtain the first term in the assertion of the Lemma. The second term, by \eqref{1.14}, 
\eqref{1.6old} and  the equivariance of $P_{j+1}$ gives the second term in the assertion.
\end{proof}

\begin{lem} \label{lem 2.3}
Let $\varrho$ be a filiform representation.  For all $1 \leq j \leq m-1,$ with the constant $w$  of Lemma \ref{lem 2.1}, 
$$
P_{j+1}\iota D^{(z)} \tilde{\varrho}_j(Y(g,z)) = -\tfrac{w}{2} \tilde{\varrho}_{j+1} (Y(g,z))\tilde{\varrho}_j(Y(g,z)).
$$ 
\end{lem}

\begin{proof}
We abreviate $Y=Y(g,z)$. Using Lemma  \ref{lem 1.6}, the linearity of $\tilde{\varrho}_j$, and that $\tilde{\varrho}_j(Y)$ is the restriction to $V_{j-1}$ of the representation $\varrho_{(1,\ldots,1)}$ of $\mathfrak k^\mathbb C + \mathfrak p^-$ 
we find 
\begin{eqnarray*}
\lefteqn{P_{j+1}\iota D^{(z)} \tilde{\varrho}_j(Y) = \frac{1}{2} P_{j+1}\iota \tilde{\varrho}_j([Y[Y,\cdot]])}\\
 &=& \frac{1}{2}P_{j+1}\iota  \tilde{\varrho}_j(Y) \varrho_{j-1}^0([Y,\cdot]) - \frac{1}{2} P_{j+1}\iota  \varrho_{j+1}^0([Y,\cdot]) \tilde{\varrho}_{j}(Y).
\end{eqnarray*}
Now $\iota$ commutes with $\tilde{\varrho}_j(Y)$, and from \eqref{2.4} we have $P_{j+1} \tilde{\varrho}_j=  \tilde{\varrho}_{j+1} P_j.$  Using this and Lemma \ref{lem 2.1}, we get that the first term equals
$$
\tfrac{1}{2}c_j(\lambda) \tilde{\varrho}_{j+1}(Y) \tilde{\varrho}_j(Y).
$$
For the second term, Lemma \ref{lem 2.1} immediately gives 
$$
-\frac{1}{2}c_{j+1}(\lambda)  \tilde{\varrho}_{j+1}(Y) \tilde{\varrho}_j(Y).
$$ 
The statement now follows from \eqref{2.8}.
\end{proof} 

For an indecomposable filiform Hhhvb $E^\varrho$ as described above, we will use the notation $E^{y,\lambda}$. Writing $0=(0,\ldots,0),$ $E^0$ makes sense, it is the direct sum of the irreducible  factor bundles in the composition series of $E^y.$

We denote by $U^y$ resp. $U^0$ the $\tilde{G}$-action on the sections of $E^y$ and $E^0$ defined by \eqref{1.13}; we observe that in the case of $U^0$ the second factor in \eqref{1.13-} is identically the identity.

If $f\in \Hol(\mathcal D, V),$ we write $f_j$ for the component of $f$ in $V_j,$ that is, the projection of $f$ onto $V_j.$ We continue using the notations introduced up to this point.

\begin{thm} \label{2.4}
Let $\varrho$ be a filiform representation of $\mathfrak k^\mathbb C + \mathfrak p^+,$ and $E^y$  the holomorphically induced vector bundle. Suppose that $\lambda\in \mathbb  R$ is regular in the sense that  
\begin{align*}
c_{ij}&= \tfrac{2^{i-j}}{(i-j)!} \prod_{k=1}^{i-j} \big (c_{j+1}(\lambda) + c_{j+k}(\lambda) \big )^{-1}\\
&=\tfrac{1}{(i-j)!} \prod_{k=1}^{i-j} \big \{ u + (j+\tfrac{k-1}{2})w -  \tfrac{\lambda}{2n} \big \}^{-1}
\end{align*}
is meaningful for $0\leq j <  i \leq m.$
Then the operator $\Gamma: \Hol(\mathcal D, V) \to \Hol(\mathcal D, V)$ given by 
$$
(\Gamma f_j)_\ell  = \begin{cases}  c_{\ell j}\, y_\ell \cdots y_{j+1} (P_\ell \iota D)  \cdots (P_{j+1} \iota D) f_j & \mbox{\!\!\!\rm if~} \ell > j,\\
f_j \:\mbox{~\rm if~} \ell = j, \\
0 \:\:\:\mbox{~\rm if~}\ell< j&
\end{cases}$$
intertwines the actions $U^0$ and $U^y$ of $\tilde{G}$ on the trivialized sections of $E^0$ and $E^y$.
\end{thm}
\begin{proof}
It is helpful to think of $f$ as a (column) vector with entries $f_j$ and of $\Gamma$ as a lower triangular matrix. 

We must show that $\Gamma$ intertwines the actions of $\tilde{G}$ via the multipliers $\varrho^0(\tilde{k}(g,z))$ respectively $\varrho^0(\tilde{k}(g,z)) \varrho^-(\exp Y(g,z)).$ The first multiplier acts diagonally.  For the second multiplier, we observe that $\varrho^-(Y)$ acts by a subdiagonal matrix
$$
\varrho^-(Y)_{j,k} = \delta_{j-1,k} y_j \tilde{\varrho}_j(Y).
$$ 
Hence, by exponentiation, $\varrho^-(\exp Y(g,z))$ is lower triangular and for $i\geq j,$
\begin{align}\label{i}
\varrho^-(\exp Y(g,z))_{i,j} &= \exp \big (\varrho^-(Y(g,z))\big )_{i,j}\nonumber\\
&=\tfrac{1}{(i-j)!} y_i \cdots y_{j+1} \tilde{\varrho}_i \big (Y(g,z)\big ) \cdots \tilde{\varrho}_{j+1}\big (Y(g,z) \big ).
\end{align}
The intertwining property to be proved is 
\begin{equation}\label{ii}
\Gamma \big (\varrho^0(\tilde{k}(g,z)^{-1})f(g z)\big ) = \varrho^-(\exp -Y(g,z)) \varrho^0(\tilde{k}(g,z)^{-1})(\Gamma f)(gz)
\end{equation}
We set $f=f_j$ (thinking of $f$ as a ``vector'' whose only non-zero component is the $j^{\rm th}$ one) and write the $\ell^{\rm th}$ component of the left hand side, for $\ell \geq j,$
\begin{equation}
c_{\ell,j} y_\ell \cdots y_{j+1} (P_\ell\iota D) \cdots (P_{j+1} \iota D) \Big (\varrho_j^0(\tilde{k}(g,z)^{-1}) (f_j\circ g) (z) \Big ). 
\end{equation}
Using the abbreviation 
$$
F^{(i)} = (P_i\iota D) \cdots (P_{j+1} \iota D) f_j,
$$ 
the corresponding component on the right hand side of \eqref{ii} is  
\begin{equation}
\sum_i \tfrac{(-1)^{\ell-i}}{(\ell -i)!} y_\ell \cdots y_{i+1} \tilde{\varrho}_\ell(Y(g,z))\cdots \tilde{\varrho}_{i+1}(Y(g,z)) c_{i,j} y_i \cdots y_{j+1} \varrho^0_i(\tilde{k}(g ,z)^{-1} ) F^{(i)}(gz)
\end{equation}
with the terms being non-zero only for $i\geq j.$ So, verifying \eqref{ii} amounts to verifying 
\begin{multline}\label{v}
c_{\ell,j}(P_\ell\iota D) \cdots (P_{j+1} \iota D) \Big (\varrho_j^0(k(g,z)^{-1}) f_j(gz) \Big ) \\
= \sum_{i=j}^\ell \tfrac{(-1)^{\ell-i}}{(\ell -i)!} c_{i j}\tilde{\varrho}_\ell(Y(g,z))\cdots \tilde{\varrho}_{i+1}(Y(g,z))\varrho^0_i(\tilde{k}(g,z)^{-1}) F^{(i)}(gz)
\end{multline}
for all $\ell \geq j.$ 
We prove \eqref{v} by induction on $\ell \geq j.$ For $\ell = j$ the identity is trivial. To pass from $\ell$ to $\ell + 1,$ we have to show that applying $(P_{\ell+1} \iota D)$ to the right hand side we get $\tfrac{c_{\ell,j}}{c_{\ell+1,j}}$ times the analogous expression with $\ell+1$ in place of $\ell.$ Using the product rule for $\iota D,$  in a first step, we get
\begin{multline*} 
\sum_i \tfrac{(-1)^{\ell-i}}{(\ell -i)!} c_{i,j}P_{\ell+1} \Big \{ \sum_{k=i+1}^\ell \tilde{\varrho}_\ell(Y(g,z)) \cdots (\iota D \tilde{\varrho}_k(Y(g,z))) \cdots \tilde{\varrho}_{i+1}(Y(g,z)) \varrho^0_i(\tilde{k}(g,z)^{-1}) F^{(i)}(gz) +\\  \tilde{\varrho}_\ell(Y(g,z)) \cdots \tilde{\varrho}_{i+1}(Y(g,z)) \iota D\big ( \varrho^0_i(\tilde{k}(g,z)^{-1}) F^{(i)}(gz) \big )\Big \}
\end{multline*}
Repeated application of \eqref{2.4} moves 
$P_{\ell+1}$ forward to give $P_{k+1}\iota D(\tilde{\varrho}_k(Y(g,z)))$ in the terms of the sum over $k$ and $P_{i+1} \iota D$ in the last factor of the  last term. At this point Lemmas \ref{lem 2.2} and \ref{lem 2.3} can be applied and give, after collecting like terms,   
\begin{multline*}
\sum_i \tfrac{(-1)^{\ell-i}}{(\ell -i)!} c_{i,j} \Big \{ -\tfrac{1}{2} (c_{i+1}(\lambda) + c_{\ell+1} (\lambda) ) \tilde{\varrho}_{\ell+1}(Y(g,z))\cdots \tilde{\varrho}_{i+1}(Y(g,z)) \varrho^0_i(\tilde{k}(g,z)^{-1}) F^{(i)}(gz)   \\
+ \tilde{\varrho}_{\ell+1}(Y(g,z))\cdots \tilde{\varrho}_{i+2}(Y(g,z))\varrho^0_{i+1}(\tilde{k}(g,z)^{-1}) F^{(i+1)}(gz)
\Big \}.
\end{multline*}
This splits naturally into two sums.  In the first sum, we slightly rewrite the coefficient in front, in the second sum, we change the summation index $i$ to $i-1,$ and obtain
\begin{multline*}
\tfrac{1}{2} \sum_{i=j}^\ell \tfrac{(-1)^{\ell-i+1}}{(\ell -i+1)!} (\ell -i+1) c_{i,j} (c_{i+1}(\lambda) + c_{\ell+1} (\lambda) )
\big ( \tilde{\varrho}_{\ell+1}(Y(g,z))\cdots \tilde{\varrho}_{i+1}(Y(g,z)) \varrho^0_i(\tilde{k}(g,z)^{-1} \big ) F^{(i)}(gz)\\
+ \sum_{i=j+1}^{\ell+1} \tfrac{(-1)^{\ell-i+1}}{(\ell -i+1)!} c_{i-1,j} \tilde{\varrho}_{\ell+1}(Y(g,z))\cdots \tilde{\varrho}_{i+1}(Y(g,z))\varrho^0_i(\tilde{k}(g,z)^{-1}) F^{(i)}(gz).
\end{multline*}
This can be written as a single sum over $i$ from $j$ to $\ell+1.$ (The two extra terms at the ends are $0$ since we may set $c_{j-1,j}=0.$) This sum will be $\tfrac{c_{\ell,j}}{c_{\ell+1,j}}$ times the $(\ell+1)$- analogous term of the right hand side of \eqref{v}, i.e. our induction will be complete if all corresponding coefficients agree, i.e. if 
$$
\tfrac{\ell-i+1}{2}c_{i,j}(c_{i+1}(\lambda) + c_{\ell+1} (\lambda) ) + c_{i-1,j} = \tfrac{c_{\ell,j}}{c_{\ell+1,j}}c_{ij}
$$
for all $0\leq i \leq j\leq \ell.$ One can easily verify that these identities follow from \eqref{2.8} finishing the proof.
\end{proof} 
To pass to more general indecomposable Hermitizable $(\varrho,V),$ it is useful to first consider the ``filiform with multiplicities" case, where for $0\leq j \leq m,$ 
$$
V_j = \mathbb C^{d_j}\otimes W_j,\,\, \varrho_j^0 = \chi_{\lambda-j}I_{d_j}\otimes \alpha_j
$$  
with irreducible representations $(\alpha_j,W_j)$ of $\tilde{K}_{\rm ss}.$ Now $\oplus W_j$ is filiform, and $\tilde{\varrho}_j(Y):W_{j-1} \to W_j$ is defined as before. Since $\varrho^-_j$ has to satisfy \eqref{1.1concise}, it follows by the same argument as before that  
$$
\varrho_j^-(Y) = y_j\otimes\tilde{\varrho}_j(Y) \:\:\: (Y \in \mathfrak p^-)
$$
with some linear transformation $y_j\in \Hom(\mathbb C^{d_{j-1}}, \mathbb C^{d_j}).$ We are using here the natural identification $\Hom(\mathbb C^{d_{j-1}}, \mathbb C^{d_j})\otimes \Hom(W_{j-1}, W_j) = \Hom(\mathbb C^{d_{j-1}} \otimes W_{j-1}, \mathbb C^{d_j}\otimes W_j)$. 

Hence the formula \eqref{i} remains correct after putting a $\otimes$ symbol after the product $y_i \cdots y_{j+1}.$  We define $\Gamma$ as in Theorem \ref{2.4}, again putting the $\otimes$ symbol after the $y$- product. The intertwining property of $\Gamma$ follows as before from \eqref{v}, in which the $y_k$-s play no role. So, the analogue of Theorem \ref{thm 2.5n} holds.

The possible Hermitian structures, as indicated at the beginning of this section, are given by positive definite linear transformations $H_j$ on $\mathbb C^{d_j},\,(0\leq j \leq m)$.

Now we consider the most  general indecomposable Hermitizable $(\varrho,V).$ Here for each $0\leq j \leq m$ there is a set $A_j$ of inequivalent irreducible representations of $\tilde{K}_{\rm ss}$ such that 
\begin{align*}
V_j &=\oplus_{\alpha\in A_j}V_j^\alpha, & \varrho_j^0 &=\oplus_{\alpha\in A_j}\varrho_j^{0\alpha}\\
V_j^\alpha &= \mathbb C^{d_{j \alpha}} \otimes W^\alpha, & \varrho_j^{0\alpha} &= \chi_{\lambda -j} I_{d_{j \alpha}}\otimes \alpha. 
\end{align*}
The possible inner products on $V^\alpha_j$ are given by  positive definite linear transformations $H_j^\alpha.$ 

We call a sequence $(\alpha_j,\ldots, \alpha_i)$ with $\alpha_k$ in $A_k$ $(j \leq k \leq i)$ \emph{admissible} if each $\alpha_k$ is contained in $\Ad^\prime_{\mathfrak p^-} \otimes \alpha_{k-1}$ for $j+1 \leq k \leq i.$ When a two term sequence $(\alpha, \beta)$ is admissible, we have the equivariant map $P_{\alpha, \beta}:\mathfrak p^-\otimes W^\alpha \to W^\beta$ and $\tilde{\varrho}_{\alpha \beta}(Y)$ for $Y\in \mathfrak p^-$ as in \eqref{2}.  As in the ``filiform with multiplicity" case, it follows that 
$$
{\varrho}^-_j(Y) = \oplus\big ( y_j^{\alpha\beta}\otimes\tilde{\varrho}_{\alpha\beta}(Y)\big )
$$
with some $y_j^{\alpha\beta}$ in $\Hom(\mathbb C^{d_{j\alpha}}, \mathbb C^{d_{j\beta}}),$ the direct sum taken over all admissible pairs $(\alpha,\beta)$ in $A_{j-1}\times A_j.$ 

Knowing the set $y=\{y_j^{\alpha\beta}\}$ implies knowing the sets $A_j$ and the multiplicities $d_{j \alpha }$. So, our general irreducible Hhhvb is determined by $y$ and $\lambda$; we may denote it by $E^{y,\lambda}$ or just $E^y$ when $\lambda$ is taken for granted. Changing all the $y$ -s in $E^y$ to $0$ we get a Hhhvb $E^0$ which is the direct sum of the factors in a composition series of $E^y$:
\begin{align} 
E^0 &= \oplus_{j=0}^m E_j \label{eqn:2.22}\\
E_j &= \oplus_{\alpha \in A_j} d_{j\alpha} E^{\alpha, \lambda-j}. \label{eqn:2.23}
\end{align} 
(By $dE$ we mean the direct sum of $d$ copies of $E$.)

From \eqref{2.2}, it follows that 
\begin{equation}\label{2.21}
y_{j+1}^{\beta\gamma}y_j^{\alpha\beta} = 0
\end{equation}
unless \eqref{3} is satisfied for $P_{\alpha\beta}$ and $P_{\gamma\beta}$ in place of $P_j$ and $P_{j+1}.$ If \eqref{3} is satisfied, we say that $(\alpha, \beta, \gamma)$ is a \emph{filiform sequence}; we call an admissible sequence 
$\boldsymbol \alpha=(\alpha_j, \ldots, \alpha_i)$ $(\alpha_k \in A_k)$ of any length filiform if it has only two terms or if every three term part of it is filiform. This is  equivalent to saying that $W^{\boldsymbol \alpha} = W^{\alpha_j}\oplus \cdots \oplus W^{\alpha_i}$ with $\mathfrak p^-$ acting via $\tilde{\varrho}_{\alpha_j \alpha_{j+1}} \oplus \cdots \oplus  \tilde{\varrho}_{\alpha_{i-1} \alpha_{i}}$ is a filiform representation. 

For $W^{\boldsymbol \alpha},$ Lemma \ref{lem 2.1} holds and defines the numbers $c_\ell^{\boldsymbol \alpha}(\lambda)$ (which depend only on $\alpha_{\ell - 1}$ and $\alpha_\ell,$ not on other terms of $\boldsymbol \alpha$), 
$u^{\boldsymbol \alpha},$ $w^{\boldsymbol \alpha}.$  Then we define 
\begin{equation}\label{2.22}
c_{ij}^{\boldsymbol \alpha} = \tfrac{2^{i-j}}{(i-j)!} \prod_{k=1}^{i-j} \big (c_{j+1}^{\boldsymbol \alpha}(\lambda) + c_{j+k}^{\boldsymbol \alpha}(\lambda)\big )^{-1}
\end{equation}
for all $\lambda \in \mathbb R$ that are regular for $\boldsymbol \alpha,$ in the sense that the right hand side is meaningful.

We introduce some abbreviations. For a filiform ${\boldsymbol \alpha}=(\alpha_j,\ldots, \alpha_i)$ we write 
\begin{align*}
y^{\boldsymbol \alpha}&= y_i^{\alpha_{i-1},\alpha_i}\cdots y_{j+1}^{\alpha_j,\alpha_{j+1}}\\
\varrho^{\boldsymbol \alpha}(Y)&= \tilde{\varrho}_i^{\alpha_{i-1},\alpha_i}(Y)\cdots \tilde{\varrho}_{j+1}^{\alpha_j,\alpha_{j+1}}(Y)\\
D^{\boldsymbol \alpha}&=(P_{\alpha_{i-1},\alpha_i}\iota D) \cdots (P_{\alpha_j,\alpha_{j+1}}\iota D)
\end{align*} 
Furthermore, for $\alpha \in A_j$, $\beta \in A_i$, ($j < i)$,   we denote by $A_{ji}(\alpha,\beta)$ the set of all filiform sequences (associated to $\varrho$) $\boldsymbol \alpha = (\alpha_j,\ldots,\alpha_i)$ such that $\alpha_j=\alpha, \alpha_i=\beta.$


For $f\in \Hol(\mathcal D, V),$  we write $f_j^\alpha$ for its projection onto $V_j^\alpha.$  

\begin{thm}\label{thm 2.5n}
Let $(\varrho,V)$ be indecomposable Hermitian and let $E^\varrho=E^y$ be the corresponding Hhhvb. Suppose that $\lambda\in \mathbb R$ is regular for every ${\boldsymbol \alpha}$ occurring in $\varrho.$  Then the operator $\Gamma^{y,\lambda}: \Hol(\mathcal D, V) \to \Hol(\mathcal D, V)$ given by 
$$
(\Gamma^{y,\lambda} f^\alpha_j)_\ell^\beta  = \begin{cases} \sum_{\boldsymbol \alpha\in A_{j\ell}(\alpha,\beta)}  c_{\ell j}^{\boldsymbol \alpha}\, y^{\boldsymbol \alpha} \otimes D^{\boldsymbol \alpha} f_j^{\alpha} & \mbox{\!\!\!\rm if~} \ell > j,\\
f_j^{\alpha} \:\mbox{~\rm if~} \ell = j, ~\mbox{\rm and}~\beta =\alpha\\
0 \:\:\:\mbox{\rm ~otherwise~}&
\end{cases}$$
intertwines the actions of $\tilde{G}$ on the trivialized sections of $E^0$ and $E^\varrho$.
\end{thm} 
\begin{proof}
We have to prove \eqref{ii} in our more general situation. $\varrho^0(\tilde{k}(g,z)^{-1})$ acts diagonally by $\chi_{\lambda-j} I \otimes \alpha_j$ on each $V_j^\alpha.$ For the other multiplier, we use \eqref{i} and get, for  the $V_i^\gamma$- component of the image of any $v_j \in V_j^\alpha$, $(j < i)$
$$
\big ( \varrho^-(\exp-Y(g,z)) v_j^\alpha\big )_i^\gamma = \tfrac{(-1)^{i-j}}{(i-j)!} \sum_{{\boldsymbol \alpha}\in A_{ji}(\alpha,\gamma)}  y^{\boldsymbol \alpha} \otimes \varrho^{\boldsymbol \alpha}(Y(g,z)) v_j^\alpha.
$$
We write down the $V^\beta_\ell$- component of the left hand side of \eqref{ii} applied to $f_j^\alpha,$ for $\ell > j:$ 
\begin{equation}\label{2.23}
\sum_{{\boldsymbol \alpha}\in A_{j\ell}(\alpha,\beta)}c_{\ell j}^{\boldsymbol \alpha}\, y^{\boldsymbol \alpha} \otimes D^{\boldsymbol \alpha}\big (\chi_{\lambda-j} I \otimes \alpha_j (\tilde{k}(g,z)^{-1}) f^\alpha_j(g z)\big )
\end{equation}
and the  corresponding right hand side:
\begin{equation} \label{2.24}
\sum_{i=j}^\ell \tfrac{(-1)^{\ell -i}}{(\ell - i)!} \sum_{\gamma\in A_i}\Big ( \sum_{\boldsymbol \sigma \in A_{i \ell}(\gamma,\beta)}   y^{\boldsymbol \sigma}\otimes \varrho^{\boldsymbol \sigma} (Y(g,z))\,\,\Big ) \Big (\,\chi_{\lambda-i} I \otimes \gamma  (\tilde{k}(g,z)^{-1}) \sum_{\boldsymbol \tau \in A_{j i}(\alpha, \gamma)} c_{ij}^{\boldsymbol \tau} y^{\boldsymbol \tau}\otimes (D^{\boldsymbol \tau}f_j^\alpha)(gz) \,\,\Big )
\end{equation}
By \eqref{2.21} we have $y^{\boldsymbol \sigma}y^{\boldsymbol \tau}=0$ unless the sequence $({\boldsymbol \tau},{\boldsymbol \sigma})$ (i.e $\boldsymbol \sigma$ following ${\boldsymbol \tau}$) is filiform. The triple sum gives then all sequences in $A_{j\ell}(\alpha,\beta)$ exactly once. 
$\gamma$  is $\alpha_i$ and $\boldsymbol{\tau, \sigma}$ are the parts of $\boldsymbol \alpha$ up to resp. beyond $\alpha_i.$ So \eqref{2.23} is equal to 
\begin{equation}\label{2.26}
\sum_{\boldsymbol \alpha \in A_{j \ell}(\alpha, \beta)} \sum_{i=j}^\ell \tfrac{(-1)^{\ell -i}}{(\ell -i )!} c_{i j}^{\boldsymbol \alpha} y^{\boldsymbol \alpha}  \otimes \varrho^{\boldsymbol \sigma}(Y(g,z)) (\chi_{\lambda-i}\otimes \alpha_i)(\tilde{k}(g,z)^{-1}) (D^{\boldsymbol \tau} f_j^\alpha)(g z).   
\end{equation}
By \eqref{v}, which was proved in Theorem \ref{2.4}, the terms of \eqref{2.26} for each $\boldsymbol \alpha$ agree with the corresponding term in \eqref{2.23}, finishing the proof.  

\end{proof}
\begin{rem}
\begin{enumerate}
\item When looking for examples of Hermitian representations of $\mathfrak k^\mathbb C+\mathfrak p^-$ one finds a large class by taking representations of the simple Lie algebra $\mathfrak g^\mathbb C$ and restricting them to  
$\mathfrak k^\mathbb C+\mathfrak p^-.$ By an unpublished result of R. Parthasarathy \cite{RP} taking sub-quotients of such (not necessarily irreducible) representations and tensoring with one dimensional representations one gets all possible indecomposable Hermitian representations. The proof (originally given for $K^\mathbb C P^-$) uses the Borel-Weil theorem.
\item A  filiform representation can be constructed by taking an arbitrary irreducible $(\varrho^0,V_0)$ and defining inductively $(\varrho_j^0,V_j)$ as the irreducible piece of $\mathfrak p^-\otimes V_{j-1}$ 
whose highest weight is the sum of 
 the highest weights of $\Ad_{\mathfrak p^-},$ resp. $\varrho_{j-1}^0 $ (the ``Cartan product"). To see that \eqref{3} holds in this case, we note that $\mathfrak p^-\otimes \mathfrak p^-\otimes V_{j-1}$ now contains $V_{j+1},$ whose highest weight is $2\beta+\Lambda$   with multiplicity $1.$  So the two sides of \eqref{3}, which are images of $Y\otimes Y^\prime\otimes v$ under $\tilde{K}$- equivariant maps must coincide up to constant. The case $Y=Y^\prime$ shows that the constant is $1.$ 
\item In the case where $\mathcal D$ is the one variable disc, we have 
$\tilde{K}_{\rm ss}=\{1\}$ and its only representation is the trivial one. In this case Theorem \ref{thm 2.5n} reduces to \cite[Theorem 3.1]{KM}.  
\item When $\mathcal D$ is the unit ball in $\mathbb C^2,\,\tilde{K}_{\rm ss}$ is ${\rm SU}(2)$ whose irreducible representations we denote by $\tau_0, \tau_1, \ldots$ (with $\dim \tau_{k} = k+1$). We have $\Ad_{\mathfrak p^-}^\prime\cong \tau_1,$ and by the Clebsch-Gordan formula $\Ad^\prime_{\mathfrak p^-}\otimes\tau_k = \tau_{k-1} \oplus \tau_{k+1}.$ The construction in ($2$) together with a $\chi_\lambda$  gives the filiform sequences $\tau_j, \ldots , \tau_\ell$ (for any fixed $(0\leq j \leq \ell)$) and the corresponding filliform representations. The contragradients of these namely $\tau_\ell, \ldots , \tau_j$ are also filiform.  There are no others since one can easily show  that $(\tau_k,\tau_{k\pm 1}, \tau_k)$ is never a filiform sequence. 
\item In general, the irreducible Hhhvb-s $E^{y,\lambda}$, $E^{y^\prime, \lambda}$ are isomorphic if there exists a family of invertible linear transformations $\{a_j^\alpha\}$ such that ${y^{\prime}}^{\alpha\beta} = (a_j^\beta)^{-1} y_j^{\alpha\beta} a_{j-1}^\alpha$. Given also the Hermitian structures $H$ (resp. $H^\prime$), they are isomorphic if, in addition, $H_j^{\prime \alpha} = (a_j^\alpha)^\star H_j^\alpha a_j^\alpha$.    
\end{enumerate}
\end{rem}

\section{Hilbert spaces of sections}

Some relatively simple known facts about vector-valued reproducing kernel spaces are fundamental for this section. We start by listing these in the exact form we will need them.  They are not difficult to prove in the order given.  Most of the statements can be found, for instance in \cite[Chapter I]{N}, although with rather different notations. 

We consider (complete) Hilbert spaces $\mathcal H\subseteq \mathfrak F(\D,V),$ where $V$ is  a finite dimensional Hilbert space and $\mathfrak F(\D,V)$ is the set all $V$ - valued functions on a set $\D.$ The inner product on $\mathcal H$ is denoted by $(\cdot | \cdot),$ on $V$ by $\langle \cdot , \cdot \rangle.$  The adjoint of an element $A$ in $\Hom(V,V)$
(and more generally, of a linear transformation between two finite dimensional Hilbert spaces) is denoted $A^\#,$ while $*$ is used in the case of infinite dimensions, e.g. for $\mathcal H.$   

For $v\in V,$  we define $v^\#$ in the linear dual of $V$ by $\langle \cdot, v\rangle.$ (This is actually the adjoint if we identify $v$ with the map $z \to z v$ in $\Hom(\mathbb C, V).$) We have $v^\# A^\# = (A v)^\#$ for $A$ as above. 

If $K(z,w)$ is a ``kernel'', i.e., a $\Hom(V,V)$ - valued function of $z$ and $w$ in $\D,$ we write, for any $v\in V,$ 
\begin{eqnarray}
K_w(z)=K(z,w),\\ \big ( K_w v\big )(z) = K_w(z)v.
\end{eqnarray}
Given $\mathcal H\subseteq \mathfrak F(\D,V),$ we say $K$ (or $K(z,w)$) is a reproducing kernel for $\mathcal H$ if $K_wv \in \mathcal H$ for all $w$ and $v,$ and if 
\begin{equation}
(f | K_wv) = \langle f(w),v\rangle
\end{equation}
for all $f\in \mathcal H.$ (It is obvious that $K$ is unique and that it exists if and only if the ``evaluation maps'' $\rm{ev}_w f = f(w)$ from $\mathcal H$ to $V$ are continuous for all $w.$ As a linear map  $V\to \mathcal H,\,\, K_w$ is just the adjoint $\rm{ev}_w^*.$)  The reproducing kernel is positive definite, denoted $K\succ 0,$ in the sense that 
$$
\sum_{j,k} \langle K(z_j,z_k)v_k, v_j\rangle \geq 0
$$
for any $z_1,\ldots ,z_N$ in $\D$ and $v_1, \ldots, v_N$ in $V.$ In particular, this implies $K(z,w)^\# = K(w,z).$ For any two kernels we write $K_0 \prec K_1$ if $K_1 - K_0 \succ 0.$  

We mention that if $\{e_\nu\}$ is any orthonormal basis for $\mathcal H$ and $K$ is the reproducing kernel, then 
\begin{equation} \label{K(3.4)}
K(z,w) = \sum_\nu e_\nu(z) e_\nu(w)^\#
\end{equation}
the sum being convergent both in $\Hom(V,V)$ and also in $\mathcal H$ when it is applied to a $v\in V,$ and regarded as a function of $z$ with $w$ fixed. 

Suppose $T:\mathcal H \to \mathfrak F(\D,V)$ is a linear transformation. Then $TK_w$ (for the reproducing kernel $K,$ or any other kernel $K_w$ such that $K_wv$ is in $\mathcal H$ for all $w,v$)  is naturally defined by 
$$
(TK_w) v = T (K_w v)\,\, (v\in V).
$$
Depending on  the context, we will also use the notation $T^{(z)}K(z,w)$ and $T_1K(z,w)$ for $\big (TK_w\big )(z)$ to indicate that the operator is applied to $K(z,w)$  as a function of $z,$ i.e. the first variable, with $w$ held fixed. 

For every $f\in \mathcal H$ we define $f^\#$ by $f^\#(z) = f(z)^\#$  $(z\in \D).$ For a linear transformation $T$ of $\mathcal H$ we define $T^\#$ by 
\begin{equation}
T^\# f^\# = \big ( Tf\big )^\#.
\end{equation}
Now ${T^{(w)}}^\#K(z,w)$ makes sense, by 
\begin{equation}\label{eq:3.6}
{T^{(w)}}^\# K(z,w) = {T^{(w)}}^\# K(w,z)^\# = \big ( T^{(w)} K(w,z) \big )^\#.
\end{equation}

Note that if $A$ is a $\Hom(V,V)$ - valued function on $\D,$ and $T_A$ on $\mathcal H$ is defined by $\big (T_Af\big ) (z) = A(z) f(z)$ ( a kind of multiplication operator), then 
$$\big (T_A^\#\big )^{(w)} K(z,w) = K(z,w) A(w)^\#,$$
and using a natural abbreviation, 
$$T_AT_A^\# K = A K A^\#. $$
For easier reference we give numbers to the following  statements.
\begin{prop}\label{prop:3.1}
Suppose that $\mathcal H\subseteq \mathfrak F(\D,V)$ has reproducing kernel $K,$ that $V_1$ is another finite dimensional Hilbert space and $T:\mathcal H \to \mathfrak F(\D,V_1)$ is a linear map such that $f\mapsto (Tf)(z)$ is bounded for every $z\in \D.$ Then $\ker T$ is closed and the range $T\mathcal H$  with the Hilbert space structure of $\mathcal H/\ker T$ transferred to it via $T$ has $T^{(z)} {T^{(w)}}^\#K(z,w)$ as its reproducing kernel. 
\end{prop}
This can be proved e.g. from \eqref{K(3.4)}. 
\begin{prop}\label{prop:3.2}
Suppose that in addition to the hypothesis of Proposition \ref{prop:3.1} there is given a Hilbert space $\mathcal H_1\subseteq \mathfrak F(\D,V_1)$ with reproducing kernel $K_1.$ Then $T$ maps $\mathcal H$ into $\mathcal H_1$ and is bounded by $c>0$ if and only if    
\begin{equation}
T^{(z)} {T^{(w)}}^\# K(z,w) \prec c^2 K_1(z,w).
\end{equation}
\end{prop}
A proof can be based  on the preceding proposition. 

An immediate consequence of Proposition \ref{prop:3.1} is the following well known fact. 
\begin{rem} 
If $G$ is a group of transformations of $\D$ and and $m(g,z)$ is a $\Hom(V,V)$ - valued multiplier (i.e. $m(gg^\prime, z) = m(g,g^\prime z) m(g^\prime, z)$ for all $g,g^\prime, z.$) and $\mathcal H$ has reproducing kernel $K,$ then $U_g$ defined by 
\begin{equation}\label{eq:3.8}
\big (U_gf\big )(z) = m(g^{-1}, z)^{-1} f(g^{-1}z)
\end{equation}
preserves $\mathcal H$ and is a unitary representation on it if and only if  $K$ is quasi-invariant, i.e. 
$$
K(gz,gw) = m(g,z)K(z,w)m(g,w)^\#
$$
for all $g,z,w.$   
\end{rem}
\begin{prop}\label{prop:3.3}
If $\D$ is a domain $\mathbb C^n$ and $\mathfrak F(\D,V)$ is changed in the statements to $\Hol(\D,V),$ the holomorphic $V$ - valued functions, and if $T$ is a holomorphic differential operator, then the hypothesis of Propositions \ref{prop:3.1} and \ref{prop:3.2} about the boundedness of $f \mapsto Tf (z)$ are  automatically satisfied.   
\end{prop}
This follows from the Cauchy estimates. 

We turn to the main subject of this section.  Given an indecomposable $E^\varrho=E^y$ as in Theorem \ref{thm 2.5n}, a \emph{regular unitary structure} on it is  
a Hilbert space $\mathcal H \subseteq \Hol(\mathcal D, V)$ with inner product invariant under $U^\varrho$ and containing the space $\mathscr P = \mathscr P(\mathfrak p^+, V)$ of all $V$ - valued polynomials. If such a structure exists, we say that $E^\varrho$ is \emph{regularly unitarizable}. 

Our first goal is to describe all regular unitary Hhhvb-s and all regularly unitary structures on them. But first we reformulate this definition in an intrinsic trivialization independent way. For this, and also for later  use,  
we recall the following facts of representation theory.

Given a continuous representation $U$ of $\tilde{G}$ on a topological vector space with some minimal good properties, the $\tilde{K}$-finite vectors, i.e. those $f$ for which $\{U_kf\mid k \in \tilde{K}\}$ span a finite dimensional space, form a dense subspace. On this subspace, $U$ induces a representation $\mathfrak u$ of $\mathfrak g$ defined by $\mathfrak u_X f = \left .\tfrac{d}{dt}\right \vert_{0} U_{\exp tX} f.$
So the $\tilde{K}$-finite vectors form  a $(\mathfrak g, \tilde{K})$ - module, i.e. a joint representation of $\mathfrak g$ and $\tilde{K}.$ (cf. \cite[Proposition  2.5]{V}).


A regular unitary structure 
can be intrinsically defined as 
a Hilbert space of holomorphic sections with inner product invariant under the action of $\tilde{G}$, such that it contains all $\tilde{K}$-finite sections.  It is equivalent to the definition first given since it is not hard to see that in the canonical trivialization the $\tilde{K}$-finite vectors are exactly the polynomials \cite[Proposition XII.2.1]{N}. These remarks also make it clear that the condition $\mathscr P \subseteq \mathcal H$ can be equivalently replaced by ``$\mathcal H$ dense in $\Hol(\mathcal D, V)$'' (in the topology of uniform convergence on compact subsets of $\mathcal D$). 

We will need the following non-trivial fact (cf. \cite[Theorem 2.12]{V}).  If $U$ is a unitary representation of $\tilde{G}$ on a Hilbert space $\mathcal H,$ and the $\tilde{K}$-finite subspace $\mathcal H_{\tilde{K}}$ is given the induced $(\mathfrak g, \tilde{K})$ - module structure, then the  $(\mathfrak g, \tilde{K})$ - sub-modules of  $\mathcal H_{\tilde{K}}$ under closure in $\mathcal H$ are in one to one correspondence with the $U$ - invariant subspaces of $\mathcal H.$ 

One important consequence of this is that if $E^\varrho$ is regularly unitarizable,  
then automatically $\mathscr P$ is dense in $\mathcal H.$

In the case of irreducible $\varrho,$ i.e. when $\varrho^0 = \chi_\lambda\otimes \alpha$ with an irreducible representation $\alpha$ of $K^\mathbb C_{\rm ss}$ and $\varrho^-=0,$ the situation is very well known, it is part of  the theory of the holomorphic discrete series of representations. For every $\alpha,$ there is a set $\mathcal W_c(\alpha)$ of the form $\lambda < \lambda_\alpha$ with $\lambda_\alpha$ explicitly known such that $E^{\alpha, \lambda}$ is regularly unitarizable if and only if $\lambda \in \mathcal W_c(\alpha)$ (cf. \cite{EHW, Jak1}). In such a case $\mathscr P$ is an irreducible $(\mathfrak{g}, \tilde{K})$ - module, hence it has a unique (up to constant)  invariant Hermitian form, which is, in this case, non-degenerate, positive and gives the inner product of the corresponding Hilbert space $\mathcal H_1^{(\alpha,\lambda)},$ which is thereby uniquely determined up to constant. 
We normalize it, as usual, by the condition $\|v\|_\mathcal H = \|v\|_V$ for $v\in V$. (Note that $v$ regarded as a constant function is in $\mathcal H_!^{(\alpha, \lambda)}$. )


Each $\mathcal H_1^{(\alpha,\lambda)}$ ($\lambda\in \mathcal W_c(\alpha)$) has a reproducing kernel $K^{(\alpha,\lambda)}$ (cf. \cite{N}, Theorem XII.2.6 and Remarks to Sec XII.2) which can be explicitly described as follows.  Exactly as in \cite[p. 64]{S} but working in $\tilde{G}_{\rm loc}$ instead of $G,$ we set, for $z,w\in \mathcal{D}$
$$
\tilde{\mathcal K}(z,w) = \tilde{k}(\exp -\bar{w},z)^{-1},
$$  
the bar denoting conjugation with respect to $\mathfrak{g}$ in $\mathfrak{g}^\mathbb C,$ and also the lift of this map to $\tilde{G}^\mathbb{C}_{\rm loc}.$ For later use we also introduce the abbreviation 
$$Y_{z,\,w} = Y(\exp -\bar{w},z)$$
so the decomposition \eqref{1.11} appears now as 
\begin{equation}\label{eq:3.9}
(\exp - \bar{w}) (\exp z) = (\exp(\exp -\bar{w})\cdot z) \tilde{\mathcal{K}}(z,w)^{-1}(\exp Y_{w,\,z}).
\end{equation}
Interchanging $z$ and $w,$  taking inverses and conjugating gives another expression for the left hand side. By uniqueness in \eqref{1.11} this implies that  
\begin{eqnarray}
(\exp - \bar{w})\cdot z &=& -\overbar{Y_{z,w}} \label{eq:3.10}\\
\tilde{\mathcal{K}}(w,z) &=& \overbar{\tilde{\mathcal{K}}(w,z)}^{\,-1}\label{eq:3.11}
\end{eqnarray} 
Also, clearly $\tilde{\mathcal{K}}(z,0) = \tilde{\mathcal{K}}(0,z)\equiv e,$ and as in \cite{S},
\begin{equation}\label{eq:3.12}
 \tilde{\mathcal{K}}(gz,gw) = \tilde{k}(g,z) \, \tilde{\mathcal{K}}(z,w) \,\overbar{\tilde{k}(g,w)}^{\,-1}.
\end{equation}
We can now verify that 
$$
K^{(\alpha, \lambda)}(z,w) = (\chi_\lambda\otimes\alpha)\big (\tilde{\mathcal K}(z,w)\big ).
$$
In fact, for $k\in \tilde{K}^\mathbb{C}$ and any Hermitian representation $\varrho^0,$ we have  $\varrho^0(\bar{k}) = \overbar{\varrho^0(k)}^{-1}.$ So if we apply $\varrho^0=\chi_\lambda\otimes \alpha$ to \eqref{eq:3.12} 
we get the quasi-invariance \eqref{eq:3.8} of $K^{(\alpha, \lambda)}$ with respect to the canonical multiplier. \eqref{eq:3.12} also shows that $K^{(\alpha, \lambda)}(z,0)\equiv 1$ which corresponds to the normalization we fixed on $\mathcal H_1^{(\alpha,\lambda)}$.  These properties characterize $K^{(\alpha,\lambda)}$.  

Changing the normalization of the invariant inner product on $\mathcal H_1^{(\alpha,\lambda)}$ we get different regular unitary structures on the $E^{\alpha,\lambda}$. We consider this question in  the greater generality of $d E^{\alpha,\lambda}$, a direct sum of $d$ copies of $E^{\alpha, \lambda}$. 

Here the space of sections is  
$$ \Hol(\mathcal D, \mathbb C^d\otimes W^\alpha)\cong \mathbb C^d \otimes \Hol(\mathcal D, W^\alpha).$$ 
(We identify the two sides. In practice, this only amounts to writing $\mathbb C^d$- valued functions in terms of a basis in $\mathbb C^d.$) The $\tilde{G}$ action is now by $I_d \otimes U^{\alpha, \lambda}$. It follows that regular unitary structures are gotten by tensoring  the inner product in $\mathcal H^{(\alpha,\lambda)}$ with an arbitrary inner product on $\mathbb C^d$. We write this latter in terms of the standard inner product of $\mathbb C^d$ as $\langle \mu \cdot, \cdot \rangle_{\mathbb C^d}$ with a positive definite linear transformation $\mu$ on $\mathbb C^d$.

We denote the regular unitary structure so obtained by $\mathcal H^{(\alpha,\lambda)}_\mu$. It is trivial to check that it has a reproducing kernel, namely, 
$$\mu^{-1} \otimes K^{(\alpha,\lambda)}(z,w).$$
This now includes the case $d=1$, where $\mu$ is scalar.


In the following we keep using the notations involved in Theorem \ref{thm 2.5n}.
We consider an indecomposable Hhhvb $E^\varrho= E^y;$ $\varrho$ is understood to determine $\lambda \in \mathbb{R}.$ We have $E^0,$ which is given by \eqref{eqn:2.22} and \eqref{eqn:2.23}.

\begin{lem} \label{lem 3.5}
If the irreducible Hhhvb $E^\varrho=E^y$ is regularly unitarizable, then so is $E^0$. 
\end{lem}
\begin{proof}
Suppose $\mathcal H$ is a regular unitary structure on $E^\varrho$ and let $\mathcal H_j = \mathcal H\cap \Hol(\mathcal D, \tilde{V}_j)$ . By $U^\varrho$ invariance of $\Hol(\mathcal D, \tilde{V}_j)$ (cf. Sec. 1), each $\mathcal H_j$ is an invariant subspace of $\mathcal H$, closed because point evaluations are continuous on $\mathcal H$. The space of sections of the bundle $E_j$ holomorphically induced by $(\varrho^0_j,0)$ is $\Hol(\mathcal D, {V}_j)$. A representation $U_j$ of $\tilde{G}$ acts on it via the multiplier $\varrho^0_j(\tilde{k}(g,z))$. The one-to-one linear map $L$ of $\mathcal H_j/\mathcal H_{j+1}$ into $\Hol(\mathcal D, {V}_j)$ defined by $L(f+\mathcal H_j) = f_j$ intertwines the quotient action of $U^\varrho$ with $U_j$. The image of $L$ (which does  contain all $V_j$-valued polynomials) with the inner product transferred from   $\mathcal H_j/\mathcal H_{j+1}$ is then a regular unitary structure on $E_j$.  Together with $E_j$ then $E_0= \oplus E_j$ is also regularly unitarizable.  
\end{proof}

The logical order would now require us to first prove Proposition \ref{prop:3.5}, because the proof of Theorem \ref{thm:3.4} uses one of its  corollaries. We invert this order because the main significance of Proposition \ref{prop:3.5} (whose proof depends only on computations done in Section 2) lies in a different direction.

\begin{thm}\label{thm:3.4}
Let $E^y=E^{y,\lambda}$ be an indecomposable Hhhvb. Then $E^y$ is regularly unitarizable if and only if  $E^0$ is, which is the case if and only if $\lambda < \lambda_\alpha+j$ for all $\alpha \in A_j,$  $0\leq j \leq m$ in the decomposition of $E^0$ as $\bigoplus d_{j \alpha} E^{\alpha, \lambda-j}.$ The regular unitary structures of $E^0$ are 
$$
\mathcal H^0_\mu = \bigoplus_{j=0}^m \bigoplus_{\alpha \in A_j} \mathcal H^{(\alpha,\lambda-j)}_{\mu_{j\alpha}},
$$
where $\mu = \{\mu_{j\alpha}\}$ and each $\mu_{j\alpha}$ is a positive definite linear transformation on $\mathbb C^{d_{j\alpha}}$. The reproducing kernel of $\mathcal H^0_\mu$ is 
$$
K^0_\mu = \oplus \oplus \mu_{j\alpha}^{-1} \otimes K^{(\alpha,\lambda-j)}.
$$
The regular unitary structures of $E^y$ are the spaces $\mathcal H^y_\mu=\Gamma^{\lambda, y}\mathcal H^0_\mu$ with  $\Gamma^{\lambda, y}$ a unitary isomorphism. The corresponding reproducing kernel is 
$$K^y_\mu = \Gamma^{\lambda, y} {\Gamma^{\lambda, y}}^\# {K^0_\mu}.$$

For a fixed $y$ (and $\lambda$) all $\mathcal H_\mu^y$ are equal as sets, and their Hilbert norms are equivalent. 
\end{thm}
\begin{proof}
 It is clear that every $\mathcal H^0_\mu$ is a regular unitary structure on $E^0$. Conversely, if $\mathcal H$ is a regular unitary structure, then it contains $\mathcal P$, which is now the direct sum of the spaces $\mathcal P^\alpha_j$ of $V^\alpha_j$ - valued polynomials. Each $\mathcal P_j^\alpha$   
is $\mathfrak u^0$ - invariant because $\mathfrak{u}^0$ (like $U^0$) respects the direct sum structure of $E^0.$ By a general result quoted above, $\mathcal H$ is therefore the direct sum of closures (in $\mathcal H$) of the spaces $\mathscr P^{\alpha, \lambda -j}.$ These closures are all of the form 
$\mathcal H^{(\alpha, \lambda-j)}_{\mu_{j\alpha}}$ because  as mentioned before, in the irreducible case the $(\mathfrak{g}, \tilde{K})$-module structure determines the inner product up to constant. 

The statement about the reproducing kernel $K^0$ is immediate from  the direct sum structure.

By Theorem \ref{thm 2.5n} and Corollary \ref{cor:1}, $\Gamma$ is an invertible map of $\Hol(\mathcal D,V)$ intertwining $U^0$ with $U^y=U^\varrho.$ Clearly, $\Gamma$ also maps $\mathscr P$ onto $\mathscr P$. So $\mathscr P \subseteq \Gamma\mathcal H^{0} \subseteq \Hol(\mathcal D,V)$ with an $U^\varrho$ - invariant inner product on $\mathcal H^{0}.$ Furthermore by Propositions \ref{prop:3.3} and \ref{prop:3.1}, $\Gamma \mathcal H^{0}$ is a complete Hilbert space with reproducing kernel $\Gamma \Gamma^\# K^0.$ 

As for the last statement, it clearly holds for every $E^{\alpha, \lambda}$, hence also for direct sums of such. So $\mathcal H_\mu^0$ is the same set for every $\mu$ and the norms are equivalent. Since $\Gamma^{y,\lambda}$ is by definition an unitary isomorphism of $\mathcal H_\mu^0$ onto $\mathcal H_\mu^y$, the same statement is true for the spaces $\mathcal H_\mu^y$.
\end{proof}

If some $E^y$ has a regular unitary structure $\mathcal H^y_\mu$, then it has a canonically associated Hermitian structure given by $H = K^y_\mu(0,0)^{-1}$.  (The inverse exists by Theorem \ref{thm:3.4} and by $K^{(\alpha, \lambda)}(0,0)= I$ for all $\alpha, \lambda$.) The regular unitary structure can be reconstructed from the Hermitian structure by the quasi-invariance of $K^y_\mu$. For Hermitian hhvb-s arising this way, we say that their metric \emph{comes from a regular unitary structure}.      
  
The fact that $U^\varrho$ on $\mathcal H^\varrho$ is equivalent to the direct sum of irreducibles is well known in the theory of the holomorphic discrete series; 
in Theorem \ref{thm:3.4} the equivalence is realized by the explicit differential operator $\Gamma.$   
  
In the second half of this section we will be looking at a filiform Hhhvb of two terms (i.e. with $m=1$). In the arguments, we need an expression for the adjoint of the map $\tilde{\varrho}_1(Y)$ defined by \eqref{2}. We derive this now as a preparation. 

Using notation of Section 2, but writing $P$ instead of $P_1,$ we define for any fixed $Y\in \mathfrak p^-$ the map $T_Y: V_0 \to \mathfrak p^-\otimes V_0$ by
$$T_Y v = Y\otimes v$$
so that $\tilde{\varrho}_1(Y) = PT_Y.$ A simple computation gives 
$$ T_Y^\# = Y^\# \otimes I_{V_0}$$ 
and therefore,  using the fact that $\iota Y= - \overbar{Y}$ for $Y\in \mathfrak p^-,$
\begin{equation}\label{eq:3.13}
\tilde{\varrho}_1(Y)^\# = \big ( Y^\# \otimes I_{V_0}\big ) P^\# = \big (B(\cdot, \overbar{Y})\otimes I_{V_0} \big ) P^\#.
\end{equation}

We consider the following situation. We set $\varrho^0_0=\chi_\lambda\otimes \alpha$ with some $\lambda \in \mathbb{R}$ and an irreducible Hermitian representation $\alpha$ of $\mathfrak{k}^\mathbb{C}_{\rm ss}.$ We take an irreducible component $\beta$ of $\Ad^\prime_{\mathfrak{p}^-}\otimes \alpha$
and set $\varrho_1^0 = \chi_{\lambda-1}\otimes \beta.$ We write $P$ for $P_{\alpha \beta}$ fixed as in Section 2, and define $\tilde{\varrho}_1(Y)$ by \eqref{2}. These data give a filiform representation with $m=1$ and we can use the corresponding notations and formulas of Section 2; in particular, we have the operator $P\iota D$ mapping sections of $E^{\alpha,\lambda}$ to sections of 
$E^{\beta,\lambda-1}.$ Based on Proposition \ref{prop:3.2} we will discuss the question whether $P\iota D$ is a bounded map of Hilbert spaces in case the two bundles are regularly unitarizable. 
\begin{prop}\label{prop:3.5}
For any $\mathcal D$ and any $\lambda, \alpha, \beta$ as above, 
\begin{equation}\label{eq:3.14}
\big (P\iota D^{(z)}\big )\big (P\iota {D^{(w)}}\big )^\# \varrho_0^0(\tilde{\mathcal{K}}(z,w)) = |c_1(\lambda)|^2 A(z,w) + \overbar{c_1(\lambda)}   \varrho_1^0(\tilde{\mathcal{K}}(z,w)),
\end{equation}
where $c_1(\lambda)$ is defined by Lemma \ref{lem 2.1} and 
\begin{equation}\label{eq:3.15}
A(z,w) = \tilde{\varrho}_1(Y_{zw})\varrho_0^0(\tilde{\mathcal{K}}(z,w)) \tilde{\varrho}_1(Y_{zw})^\# = P\big (\,Y_{z w}Y_{wz}^\#\otimes \varrho_0^0(\tilde{\mathcal{K}}(z,w))\, \big )P^\#.
\end{equation}
When $\lambda\in \mathcal W_c(\alpha),$ we have $c_1(\lambda) >0.$ 
\end{prop} 
 
\begin{proof}
Using the definition of $\tilde{\mathcal{K}},$ Lemma \ref{lem 1.6} and Lemma \ref{lem 2.1} we immediately get 
$$
\big (P\iota {D^{(w)}} \big )^\#\varrho_0^0(\tilde{\mathcal{K}}(z,w)) = \big (P\iota D^{(w)} \big )\varrho_0^0(\tilde{\mathcal{K}}(w,z)) = - \overbar{c_1(\lambda)} \varrho_0^0(\tilde{\mathcal{K}}(z,w))\tilde{\varrho}_1(Y_{wz})^\#.
$$
We set $\phi(z) = \tilde{\varrho}_1(\bar{z})^\#,$ using \eqref{eq:3.10}, the right hand side can be written 
$$ - \overbar{c_1(\lambda)} \varrho_0^0\big (\tilde{k}(\exp -\bar{w},z)^{-1}\big )(\phi(\exp -\bar{w}) \cdot z).
$$
We have to apply $P\iota D^{(z)}$ to this. We do it by applying Lemma \ref{lem 2.2}, which is certainly applicable to $F(z)=\phi(z)v$ with any $v\in V_0,$ therefore also to $\phi(z)$ by linearity in $v.$ The first term Lemma \ref{lem 2.2} gives is exactly $|c_1(\lambda)|^2$ times the first expression for $A(z,w)$ in \eqref{eq:3.15} (which is equal to the second expression by \eqref{eq:3.13}). The second term Lemma \ref{lem 2.2} gives  is 
$$ - \overbar{c_1(\lambda)} \varrho_0^0\big (\tilde{\mathcal K}(z,w) \big )
\big (P\iota D\phi \big ) (\exp -\bar{w}) \cdot z).
$$
Now $P\iota D \phi$ is constant since $\phi$ is linear in $z.$ More exactly, for any $X\in \mathfrak p^+,$ by \eqref{eq:3.13} we have 
$$
\iota D_X\phi(z) = \iota \big (B(\cdot,X)\otimes I_{V_0}\big ) P^\# = (X \otimes I_{V_0}) P^\#,$$
i.e. 
$$\iota D\phi(z) = (I_{\mathfrak p^-} \otimes I_{V_0})P^\#$$
and 
$$P \iota D\phi(z)=P(I_{\mathfrak p^-} \otimes I_{V_0})P^\# = I_{V_1}$$
finishing the proof of \eqref{eq:3.14}.
To prove the last statement: Now $\varrho^0_0(\mathcal K(z,w))$ is a positive defninite kernel, hence so is the whole left hand side of \eqref{eq:3.14}.  Therefore, for $z=0=w$ it is a positive operator. The right hand side of \eqref{eq:3.14} shows this to be equal to $\overline{c_1(\lambda)} I_{V_1}.$ Hence $c_1(\lambda) \geq 0$.  But $c_1(\lambda) = 0$ is impossible since it would imply that the left hand side is identically zero, 
which is not the case since $\varrho^0_0(\tilde{\mathcal K}(z,w))$ is the reproducing kernel of a space containing all polynomials.  \end{proof} 
\begin{cor}\label{cor:1}
If in the notation of Theorem \ref{thm 2.5n}, $E^\varrho=E^y$ is regularly unitarizable, then the corresponding $\lambda$ is regular. 
\end{cor}
\begin{proof}
By Lemma \ref{lem 3.5}, $E^0$ is regularly unitarizable, hence $\lambda < \lambda_\alpha+j$ for each $\alpha \in A_j$. This implies that each $c^{\boldsymbol \alpha}_\ell(\lambda)$ occurring in \eqref{2.22} is positive. 
\end{proof}
\begin{cor}\label{cor:2}
If $\lambda\in \mathcal W_c(\alpha),$ then the following statements are equivalent:
\begin{enumerate}
\item $A(z,w) \prec C\varrho_1^0(\mathcal{K}(z,w))$ for some $C >0.$
\item $\lambda -1 \in \mathcal W_c(\beta)$ and $P\iota D$ is a bounded linear operator from $\mathcal H^{(\alpha, \lambda)}$ to $\mathcal H^{(\beta, \lambda-1)}.$
\end{enumerate}
\end{cor}
\begin{proof}
This follows from Proposition \ref{prop:3.2}.
\end{proof}
The following is a partial reduction of the boundedness question to the ``scalar case'', i.e. the case where $\alpha = \mathbf{1}$ is the trivial representation, so $V_0 = \mathbb C.$   The corresponding vector Hhvb-s are the line bundles $L_\lambda$ already occurring in Theorem \ref{prop 1.2}. 

\begin{cor}\label{cor:3}
Suppose $\lambda_0 \in \mathcal W_c(\mathbf{1}),$ $\lambda_0 -1 \in \mathcal W_c(\Ad_{\mathfrak p^-}^\prime)$ and $\iota D$ is bounded from 
$\mathcal{H}^{(\mathbf{1},\lambda_0)}$ to $\mathcal{H}^{(\Ad_{\mathfrak p^-}^\prime, \,\lambda_0-1)}.$  Then for any irreducible $\alpha$ and any $\lambda\in \mathcal W_c(\alpha),$ we have $\lambda+\lambda_0 \in \mathcal W_c(\alpha),$ $\lambda+ \lambda_0 - 1\in \mathcal W_c(\beta),$ and $P\iota D$ is bounded from   $\mathcal{H}^{(\alpha,\lambda + \lambda_0)}$ to $\mathcal{H}^{(\beta,\lambda+ \lambda_0-1)}.$
\end{cor}
\begin{proof}
By Corollary \ref{cor:2}, the hypothesis implies 
$$
\chi_{\lambda_0}(\tilde{\mathcal{K}}(z,w)) Y_{z w}Y_{w z}^\# \prec C \chi_{\lambda_0-1}(\tilde{\mathcal{K}}(z,w))\Ad^\prime_{\mathfrak{p}^-}(\tilde{\mathcal{K}}(z,w))
$$
for some $C.$ This relation remains true after tensoring with the positive definite kernel $(\chi_{\lambda}\otimes \alpha)(\tilde{\mathcal{K}}(z,w))$ and then multiplying by $P$ on the left and $P^\#$ on the right. Hence 
$$
\chi_{\lambda+\lambda_0}(\tilde{\mathcal{K}}(z,w)) P \big ( Y_{z w}Y_{w z}^\# \otimes \alpha(\tilde{\mathcal{K}}(z,w)) \big )P^\# \prec C \chi_{\lambda+\lambda_0-1}(\tilde{\mathcal{K}}(z,w)) P (\Ad^\prime_{\mathfrak{p}^-}\otimes \alpha)(\tilde{\mathcal{K}}(z,w)) P^\#.
$$
The left hand side is just $A(z,w)$ corresponding to $\chi_{\lambda+\lambda_0}\otimes \alpha,$ and the right hand side is $\chi_{\lambda+\lambda_0-1}\otimes \beta (\tilde{\mathcal{K}}(z,w)).$ Now Corollary \ref{cor:2} implies our statement. 
\end{proof}
The last corollary shows the particular importance of the scalar case. A number of things are known about this case (cf. \cite{FK, KW}):
\begin{equation}\label{eq:3.16}
h(z,w) = \chi_{\tfrac{n}{p}}(\tilde{\mathcal{K}}(z,w))
\end{equation}
is a polynomial, holomorphic in $z,$ anti-holomorphic in $w,$ of bidegree $(r,r),$ where $r$ is the rank of $\mathcal{D};$ it can be characterized in several equivalent ways. It has an expansion
\begin{equation}\label{eq:3.17}
h(z,w) = 1-\tfrac{1}{2p} \langle z, w\rangle + \cdots ,
\end{equation}
the other terms homogeneous of  bidegree at least $(2,2).$ ($\tfrac{1}{2p}$ is the factor normalizing the inner product so that the inscribed sphere of $\mathcal{D}$ has radius $1.$ It was computed in \cite{KW}.) Furthermore, it is known that 
\begin{equation}\label{eq:3.18}
\mathcal W_c(\mathbf{1}) = \{ \lambda < - \tfrac{n}{p}(r-1)\tfrac{a}{2}\},
\end{equation}
i.e. $\lambda_1 = - \tfrac{n}{p}(r-1)\tfrac{a}{2}$. 
\begin{prop}\label{prop:3.6}
For any irreducible $\mathcal{D},$
\begin{equation} \label{eq:3.19}
\Ad_{\mathfrak p^-}(\tilde{\mathcal{K}}(z,w))=-2p\iota D^{(z)} (\iota D^{(w)})^\# \log h(z,w).
\end{equation}
\end{prop}
\begin{proof}
First we prove the quasi-invariance of the right hand side. We use the abbreviation $H(z,w) = \log h(z,w).$ Since $h(z,w)$ is quasi-invariant with holomorphic multiplier, we have 
\begin{equation}\label{eq:3.20}
(\iota D_1)(\iota D_2)^\# H(z,w) = \iota D^{(z)} (\iota D^{(w)})^\#\{H(g.z,g.w)\}.
\end{equation}
To compute the right hand side, we use a ``chain rule'' for $\iota D:$ 

\noindent
From $D^{(z)}\{f( g z)\} = (DF)(g z) \Ad_{\mathfrak{p}^+}(\tilde{k}(g,z))$ 
(cf. Remark \ref{rem:1.8}) we obtain, by \eqref{1.6old},
\begin{equation}\label{eq:3.21}
\iota D^{(z)} \{f ( g z)\} = \Ad_{\mathfrak p^-}(\tilde{k}(g,z)^{-1})(\iota Df)(g z).
\end{equation}
Applying this and using \eqref{eq:3.6} twice we obtain 
\begin{eqnarray*}
(\iota D^{(w)})^\#\{H(gz,gw)\} &=& (\iota D^{(w)}\{H(gw,gz)\})^\#\\
&=& \Ad_{\mathfrak p^-}(\tilde{k}(g,z)^{-1}) (\iota D_1 H) (g w, g z)\\
&=& \big ((\iota D_2)^\# H\big ) (g z, g w))\big (\Ad_{\mathfrak p^-}(\tilde{k}(g,z)^{-1})\big )^\#.
\end{eqnarray*}
When we apply $\iota D^{(z)}$ to this, we get 
exactly the quasi-invariance of $\iota D_1 (\iota D_2)^\# H$ with respect to $\Ad_{\mathfrak p^-}(\tilde{k}(g,z)^{-1}).$

By transitivity of $\tilde{G},$ this proves \eqref{eq:3.19} if we know that the two sides are equal for $z=0=w$.  
So, we evaluate at $z=w=0.$ On the left of \eqref{eq:3.20} we have $I_{\mathfrak p^-}$ since $\tilde{\mathcal{K}}(0,0) = e.$ On the right hand side, we use
\begin{equation}\label{eq:3.22}
-\iota D_1 (\iota D_2)^\#(\log h) = h^{-2}(\iota D h ) ((\iota D)^\# h ) - h^{-1}\iota D (\iota D)^\# h.
\end{equation}
Evaluating this at $(0,0)$ with the aid of \eqref{eq:3.17}, the first term gives $0$ and the second term gives $\tfrac{1}{2p} I_{\mathfrak p^-}$ by the easily checked identity
\begin{equation}\label{eq:3.23}
\iota D^{(z)} (\iota D^{(w)})^\# \langle z , w\rangle = I_{\mathfrak p^-}
\end{equation}
for all $z,w$ in $\mathfrak p^+.$ This completes the proof.
\end{proof}
\begin{thm}\label{thm:3.11}
Suppose $\mathcal D$ is the  Euclidean unit ball  in some $\mathbb C^n.$ If $\alpha$ is any irreducible representation of $\mathfrak{k}^\mathbb C_{\rm ss},$ $\beta$ an irreducible component of $\Ad_{\mathfrak{p}^-}^\prime \otimes \alpha,$ and $\lambda\in \mathcal W_c(\alpha),$ then $\lambda-1\in \mathcal W_c(\beta)$ and $P\iota D$ is a bounded operator from $\mathcal{H}^{(\alpha,\lambda)}$ to $\mathcal{H}^{(\beta,\lambda-1)}.$
\end{thm}

\begin{proof}
When $\mathcal{D}$ is the Euclidean ball, we have $r=1.$ We first prove the Theorem in the special case of $\alpha =\mathbf 1.$ For a more convenient parameter, we write 
$\ell = -\tfrac{p}{n} \lambda.$ By Proposition \ref{prop:3.2}, we have to prove only that if $\lambda\in \mathcal W_c(\mathbf 1),$ i.e. by \eqref{eq:3.18} if $\lambda < 0,$ i.e. if $\ell >0,$ then 
\begin{equation}\label{eq:3.24}
\iota D^{(z)} (\iota D^{(w)})^\#\{h(z,w)^{-\ell}\} \prec C h(z,w)^{-\ell} \Ad_{\mathfrak p^-}(\tilde{\mathcal K}(z,w))
\end{equation}
for some $C.$ The left hand side here equals 
\begin{equation}\label{eq:3.25}
\ell(\ell+1) h^{-\ell -2} (\iota D h)\big ( (\iota D)^\# h\big ) - \ell h^{-\ell -1} \big ( \iota D (\iota D)^\# h\big ),
\end{equation}
and we have a similar expression for the right hand side from \eqref{eq:3.19} and \eqref{eq:3.22}. It follows that choosing $C=\tfrac{\ell(\ell+1)}{p}$ we have
\begin{equation}\label{eq:3.26}
C h(z,w)^{-\ell} \Ad_{\mathfrak p^-}(\tilde{\mathcal K}(z,w)) - 
\iota D^{(z)} (\iota D^{(w)})^\# 
\{ h(z,w)^{-\ell} \} = -\ell^2 h(z,w)^{-\ell -1} \iota D(\iota D)^\# h(z,w).
\end{equation}
Since $r=1,$ the expansion \eqref{eq:3.17} ends with the term of bidegree $(1,1).$ So, because of \eqref{eq:3.23}, the right hand side of \eqref{eq:3.26} is $\tfrac{1}{2p}\ell^2 h(z,w)^{-\ell -1} I_{\mathfrak{p}^-},$ which is positive definite. This proves \eqref{eq:3.24} and the case $\alpha = \mathbf 1$  of the Theorem.

To prove the general case, suppose $\lambda\in \mathcal W_c(\alpha),$  i.e. $\lambda < \lambda_\alpha.$ We choose $\lambda^\prime$ such that $\lambda< \lambda^\prime <  \lambda_\alpha.$ Then $\lambda = \lambda^\prime + \lambda_0,$ with $\lambda_0 < 0,$ i.e. $\lambda_0 \in \mathcal W_c(\mathbf{1}).$ We now apply Corollary \ref{cor:3} with $\lambda^\prime$ in place of $\lambda,$ and get the general statement of our theorem.
\end{proof}
\section{Homogeneous Cowen-Douglas tuples}
We will be mostly concerned with the modified Cowen-Douglas class $\hat{B}_k(\mathcal D)$ which has all the basic geometric properties of the original Cowen-Douglas class but is easier to handle (see \cite[Remark p. 5]{DM}).  To recall the definitions, let $\mathcal D \subseteq \mathbb C^m$ be an arbitrary domain, and let $\mathcal H\subset \Hol(\mathcal D, \mathbb C^k)$ be a Hilbert space containing all the $\mathbb C^k$ - valued polynomials as a dense set and having a reproducing kernel $K=K(z,w).$ Suppose also that the operators $M_j,$ defined by $(M_j)f(z) = z_j f(z)$ preserve $\mathcal H$ and are bounded on it.  An $n$-tuple 
$(T_1, \ldots, T_m)$ of commuting bounded operators on any Hilbert space $H$  is said to belong to  $\hat{B}_k(\mathcal D)$ if there is a unitary isomorphism of $H$  onto $\mathcal H$ which carries $T_j$ to $M_j^*$ for each $j=1, \ldots,m.$ From now on we write $V$ in place of 
$\mathbb C^k,$ this is more convenient for what follows.  We keep assuming that $V$ has an inner product $\langle \cdot , \cdot \rangle$ (corresponding to the standard inner product in $\mathbb C^k$). 

The original Cowen-Douglas class ${B}_k(\mathcal D)$ (see \cite{CD, CS}) can be characterized in a similar way, with the requirement of density of polynomials replaced by the condition that the range of $\oplus_{j=1}^m (M_j^* -\bar{w}_j)$ mapping $\mathcal H$ into $\mathcal H \oplus \cdots \oplus \mathcal H$ is closed for all $w\in \mathcal D$. For the precise relationship between these classes, see  \cite{CS} and \cite{AS}. 

The essential fact about $\hat{B}_k(\mathcal D)$ (and about ${B}_k(\mathcal D)$ as well) is that the joint eigenspace $F_z$ of the operators $M_j^*$  for eigenvalue $\bar{z}_j$ is, for all $z\in \mathcal D,$ $k$ dimensional and equal to $\{K_z v: v \in V\}.$

The spaces   $F_z$ with their inner product inherited form $\mathcal H$ form the fibres of a Hermitian anti-holomorphic vector bundle $F$ over $\mathcal D$. 
In a natural way, the space $\mathcal H$ is the space of sections of the complex antidual $E$ of $F$, which is a Hermitian holomorphic vector bundle.  In the trivialization the fibre $E_z$ becomes $V$ with the inner product $\langle K(z,z)^{-1}\cdot, \cdot \rangle$.   

It is a fundamental result (\cite{CD, CS, AS}) that the unitary equivalence class of elements of $\hat{B}_k(\mathcal D)$ (and also of ${B}_k(\mathcal D)$) and the  corresponding isomorphism class of holomorphic Hermitian vector bundles mutually determine each other.  

When $\mathcal D$ is a bounded symmetric domain and $H$ any Hilbert space, one calls an $n$-tuple $T=(T_1, \ldots, T_n)$ of commuting bounded operatos homogeneous (cf. \cite{MS, BM}) if their joint Taylor spectrum is contained in $\overbar{\mathcal D}$ 
and for every holomorphic automorphism $g$ of $\mathcal D,$ there exists a unitary operator $U_g$ such that 
$$
g (T_1, \ldots ,T_n) = (U_g^{-1} T_1 U_g, \ldots , U_g^{-1} T_n U_g),
$$   
or more briefly
\begin{equation}\label{eqn:4.1} 
g(T)_i = U_g^{-1} T_i U_g\,\,\,\, (1\leq i \leq n). 
\end{equation}

A description of all homogeneous $n$- tuples in $B_1(\mathcal D),$ when $\mathcal D$ is a domain of classical type is in \cite{BM, MS}, for arbitrary $\mathcal D$ it is in  \cite{AZ}. 
When $\mathcal D$ is the unit disc in $\mathbb C$, a complete description of all homogeneous operators in $B_k(\mathcal D)$  is in \cite{KM}.  It is easily seen that the answer is the same  for  $\hat{B}_k(\mathcal D)$. 
For a large subclass of $B_k(\mathcal D)$ for arbitrary $\mathcal D$, there are precise results in \cite{MU}. 

%

Here we prove some simple results about the most general case, then specialize to the case of the unit ball in $\mathbb C^n$ and prove the main results of this section.  

%
\begin{thm} \label{thm:4.1}
Let $\mathcal D \subseteq \mathbb C^n$ be an irreducible bounded symmetric domain. An irreducible Hermitian holomorphic vector bundle $E$ over $\mathcal D$ corresponds to a homogeneous $n$-tuple in $\hat{B}_k(\mathcal D)$ for some $k$  if and only if it is homogeneous under $\tilde{G}$ and its Hermitian structure comes from a regular unitary structure $\mathcal H$ such that each  multiplication operator $M_i,\,1\leq j \leq n,$ preserves $\mathcal H$ and is  bounded.  
\end{thm}  
\begin{proof} For the ``if'' part: By Theorem \ref{thm:3.4} $E$ is an  $E^y$ and $\mathcal H$ is an $\mathcal H^y_\mu$ with some $y$ and $\mu$. The  polynomials are dense in $\mathcal H^y_\mu$ and it has the reproducing kernel $K^y_\mu$. By hypothesis,  $(M_1^*, \ldots , M_n^*)$ is a well-defined $n$- tuple in $\hat{B}_k(\mathcal D)$. We have to prove that $M^*$ is homogeneous; for this, it is enough to prove that $M$ is homogeneous. (As is well-known, if \eqref{eqn:4.1} holds for $T$ with $U_g$, then it holds also for $T^*$ with $U_{\overbar{g}}$, where $\overbar{g}$ is defined by $\overbar{g}(z) = \overline{g({\overbar{z})}}$ and $\overbar{z}$ is the ordinary complex conjugation.) Now $U=U^y$ acts on $\mathcal H^y_\mu$ via a multiplier $m(g,z)$, and we have 
\begin{eqnarray*}
\big ( M_iU_gf \big )(z) &=& z_im(g^{-1},z)^{-1} f (g^{-1}z)\\
\big (U_g\, g(M)_i f \big )(z) &=& m(g^{-1},z)^{-1} \big ( g(\zeta)_i f (\zeta)\big )_{\zeta = g^{-1}(z)}.
\end{eqnarray*}
The two expressions being equal, $M$ is homogeneous. 

In proving the converse, $\mathcal H$ is given with reproducing kernel $K$,  polynomials dense, and $M^*$ homogeneous.  As recalled above, the joint $\overbar{z}$- eigenspaces $F_z$ of $M^*$ form a bundle $F$ and $E$ is the anti-dual of $F$.  We must prove that $E$ is homogeneous. By \cite[Theorem 2.1]{KM}, for this it is enough to prove that for every $g\in \Aut(\mathcal D)$, there exists an automorphism of $E$ acting on $\mathcal D$ as $g$ (i.e. a bundle map $E\to E$ projecting to $g$).  For this, in turn,  it is enough to prove that $F$ has a similar property.

For all $g$ in $\Aut(\mathcal D)$ we have by hypothesis a unitary operator $U_g$ on $\mathcal H$ intertwining $M^*$ and $g(M^*)$. We show that $U_{\overbar{g}}$ maps each $F_z$ (which is a subspace of $\mathcal H$) linearly onto $F_{g(z)}$. This will give the desired automorphism of $F$. So, let $f\in F_z$, i.e. $M_i^*f = \overbar{z}_i f$, ($1\leq i \leq n$). We have 
$$M_i^* U_{\overbar{g}}f =  U_{\overbar{g}}\, \overbar{g}(M^*)_i f = U_{\overbar{g}} \overbar{g}(\overbar{z})_i f = \overline{g(z)_i}U_{\overbar{g}} f,$$ 
which shows $U_{\overbar{g}} f \in F_{g(z)}$. Doing the same with $g^{-1}$, we see that $U_{\overbar{g}}$ gives a vector space isomorphism $F_z \to F_{g(z)}$, hence an automorphism of $F$.   \end{proof}
  

The following corollary is immediate from the last statement of Theorem \ref{thm:3.4}.
\begin{cor}
For a regularly unitarizable irreducible Hhhvb the boundedness of  $M_i$  holds either for all or none of the regular unitary structures. If it holds, then the corresponding commuting tuples of multiplication operators are all similar.
\end{cor}
For general $\mathcal D$, the following proposition provides a sufficient condition.

\begin{prop}\label{prop:4.2}
Let $\mathcal D \subseteq \mathbb C^n$ be an irreducible bounded symmetric domain and let $E^y= E^{\alpha, \lambda}$ 
be an irreducible Hhhvb. We write
$$
\lambda_y = \,\, {\stackrel{\rm min}{\scriptstyle{0\,\leq j \leq m}}\,\, \stackrel{\rm min~~}{\scriptstyle{\alpha \in A_j}} \stackrel{{\textstyle \lambda_\alpha}}{}}.
$$
If $\lambda < \lambda_y - \tfrac{n}{p}(r-1)\tfrac{a}{2},$ then $E^y$ is regularly unitarizable. Each one of the Hermitian structures on $E^y$ obtained in this way corresponds to a homogeneous  tuple  in some $\hat{B}_k(\mathcal D)$. 
\end{prop}  
\begin{proof}
We choose $\lambda_0 < - \tfrac{n}{p}(r-1)\tfrac{a}{2}$ such that $\lambda^\prime = \lambda - \lambda_0 < \lambda_y$. 
So $E^{y,\lambda} = L_{\lambda_0} \otimes E^{y,\lambda^\prime}$. Choosing some $\mu = \{\mu_{j\alpha} \}$, we have the regular unitary structure $\mathcal H_\mu^{(y, \lambda^\prime)}$ on $E^{y,\lambda^\prime}$.  By \eqref{eq:3.18} we can also choose a regular unitary structure $H$ on $L_{\lambda_0}$. Now  $H\otimes \mathcal H_\mu^{(y, \lambda^\prime)}$ is a regular unitary structure on $E^{\lambda,y}$ with reproducing kernel $h(z,w)^{\tfrac{p}{n} \lambda_0}K^{(y, \lambda^\prime)}_\mu(z,w)$.  As proved in \cite{AZ}, each $M_i$ is bounded on $H$. So by our Proposition \ref{prop:3.2}, 
$$ (c^2 - z_i\bar{w}_i) h(z,w)^{\tfrac{p}{n} \lambda_0} \succ 0 $$
with some $c >0$. It follows that 
$$ (c^2 - z_i\bar{w}_i) h(z,w)^{\tfrac{p}{n} \lambda_0} K_\mu^{(y,\lambda^\prime)}(z,w) \succ 0.$$ 
which shows that $M_i$ is bounded on $H\otimes \mathcal H_\mu^{(y, \lambda^\prime)}$ (again by Proposition \ref{prop:3.2}). But then the last statement of Theorem \ref{thm:3.4} implies that $M_i$ is bounded on any of the regular structures of $E^{y,\lambda}$. 
\end{proof}

\begin{cor} \label{rem:4.3}
When $\mathcal D$ is the Euclidean unit ball in $\mathbb C^n$,  every Hermitian hhvb whose Hermitian metric comes from a regular unitary structure corresponds to a homogeneous tuple in $\hat{B}_k(\mathcal D)$ for some $k$.
\end{cor}
\begin{proof}
 The domain $\mathcal D$ is the  Euclidean unit ball if and only if $r=1$.  On the other hand by Theorem \ref{thm:3.4} we know that $E^{y, \lambda}$ is regularly unitarizable exactly when $\lambda < \lambda_y$. 
\end{proof}
In the case of a  general $\mathcal D$, these arguments leave a gap, an interval of $\lambda$ for which the  question remains open.  
%

We shall say that a homogeneous $n$-tuple in $\hat{B}_k(\mathcal D)$ is \emph{basic} if the corresponding Hermitian hhvb is of the form $E^{\alpha, \lambda}$, i.e., is induced by an irreducible  representation of  $\mathfrak k^\mathbb C + \mathfrak p^-.$ 
 If $\mathcal D$ is the unit ball, then  Corollary \ref{rem:4.3}  gives a complete characterization of these.  This result and the following theorem generalize the main results of  \cite[Theorem 4.2]{KM} to the case of the  Euclidean  unit ball in $\mathbb C^n$, ($n\geq 0$). 
 \begin{thm}\label{thm:4.2}
 If $\mathcal D$ is the  Euclidean unit ball in $\mathbb C^n,$ then every homogeneous $n$-tuple in $\hat{B}_k(\mathcal D)$ is similar to the direct sum of basic homogeneous $n$-tuples. 
 \end{thm} 
\begin{proof}
We know that the bundle for the homogeneous $n$-tuple in $\hat{B}_k(\mathcal D)$ is an $E$ with regular unitary structure. We may assume that $E$ is irreducible.
By Theorem \ref{thm:3.4} this means that $E= E^y$ with the corresponding $E^0$ a direct sum $\oplus_{j=0}^m \oplus_{\alpha \in A_j} E^{\alpha, \lambda-j}$  and $\mathcal H^y = \Gamma \mathcal H^0,$ with $\mathcal H^0 = \oplus \mathcal H^{(\alpha, \lambda-j)}.$ Each $\mathcal H^{(\alpha, \lambda-j)}$ and hence $\mathcal H^0$ is stable under $M_j\,(1\leq j \leq n)$ by Proposition \ref{thm:4.2}(a).  The essential point is that $\mathcal H^0$ and $\mathcal H^\varrho$ are the same as sets. This follows immediately from Theorem \ref{thm:3.11} and the definition of $\Gamma.$ 

Let $I$ be the identity map regarded as a linear transformation from $\mathcal H^0$ to $\mathcal H^y.$ Let $M^{(0)}_j,$ respectively $M_j^{(y)},$ be be the multiplication operators as before but regarded as operators on $\mathcal H^0$ and $\mathcal H^y$ respectively.  They are clearly intertwined by $I,$ so we have 
$$
M_j^{(y)} = I M_j^{(0)} I^{-1}
$$    
Since $M^{(0)} = (M^{(0)}_1, \ldots , M^{(0)}_n)$ is the direct sum of basic $n$-tuples coming from the Hilbert spaces $\mathcal H_j = \oplus_{\alpha\in A_j} \mathcal H^{(\alpha, \lambda -j)},$ the theorem follows.  \end{proof}

\begin{rem}
 If the analogue of Theorem \ref{thm:3.11} can be proved for more general $\mathcal D,$  then the present theorem will also hold, at least if $\lambda$ is outside the gap mentioned  after Corollary \ref{rem:4.3}.
  \end{rem}

\end{document}